\documentclass[11pt]{article}
\usepackage[utf8]{inputenc}

\usepackage{amsthm, amsfonts, amssymb, amsmath,enumitem, natbib, bbm, mathabx}
\usepackage{cases, amsthm}
\usepackage{fullpage}
\usepackage{mathtools}
\mathtoolsset{showonlyrefs}
\usepackage{ wasysym} 
\usepackage{fancyhdr}
\usepackage{graphicx}
\usepackage{bbm,bm}
\usepackage{caption}
\usepackage{subcaption}
\usepackage{algorithm}
\usepackage{algorithmic}
\usepackage{authblk}
 \usepackage{setspace}
\let\Algorithm\algorithm
\renewcommand\algorithm[1][]{\Algorithm[#1]\setstretch{1.6}}

\usepackage[colorlinks=true, pdfstartview=FitV,
linkcolor=blue, citecolor=blue, urlcolor=blue]{hyperref}
\usepackage{cleveref}
\newtheorem{definition}{Definition}[section]
\newtheorem{theorem}[definition]{Theorem}

\newtheorem{corollary}[definition]{Corollary}

\def\R{\mathbb{R}}

\def\E{\mathbb{E}}
\def\I{\mathbf{I}}

\def\N{\mathbb{N}}

\def\F{\mathcal{F}}

\def\A{\mathcal{A}}
\def\P{\mathbb{P}}
\newcommand{\gauss}{\mathcal{N}}
\newcommand{\id}{\mathbf{I}}

\newcommand{\abs}[1]{\left\lvert #1 \right\rvert}
\newcommand{\norm}[1]{\left\| #1 \right\|}
\newcommand{\cov}{\mathrm{Cov}}

\usepackage{xcolor}

\def\R{\mathbb{R}}

\def\F{\mathcal{F}}
\def\E{\mathbb{E}}

\def\A{\mathcal{A}}

\def\R{\mathbb{R}}

\def\F{\mathcal{F}}
\def\E{\mathbb{E}}

\usepackage[top=1in, bottom=1in, left=1in, right=1in]{geometry}
\title{Central Limit Theorems and Approximation Theory: Part I\footnote{This is a lightly edited version of a project report by Arisina Banerjee, submitted to Indian Statistical Institute (Kolkata) for partial fulfillment of the requirements of Masters in Statistics (M. Stat.). Thanks are due to Prof. Soumendu Sundar Mukherjee of Indian Statistical Institute (Kolkata) and to Prof. Alessandro Rinaldo of Carnegie Mellon University for insightful discussions.}}

\author[1]{Arisina Banerjee}
\author[2]{Arun Kumar Kuchibhotla}
\affil[1]{Indian Statistical Institute}
\affil[2]{Department of Statistics \& Data Science, Carnegie Mellon University}
\date{}

\begin{document}

\maketitle

\begin{abstract}  
Central limit theorems (CLTs) have a long history in probability and statistics. They play a fundamental role in constructing valid statistical inference procedures. Over the last century, various techniques have been developed in probability and statistics to prove CLTs under a variety of assumptions on random variables. Quantitative versions of CLTs (e.g., Berry--Esseen bounds) have also been parallelly developed. In this article, we propose to use approximation theory from functional analysis to derive explicit bounds on the difference between expectations of functions.
\end{abstract}

\section{Description of the problem}
\label{sec: Problem Statement}

Suppose we have a sequence of $d$-dimensional random vectors $\{X_i\}_{i\in\mathbb{N}}$ and a $d$-dimensional random vector ${Z}$ such that ${Z}$ is a gaussian random variable distributed with mean equal to  $\mathbb{E}\left(n^{-1/2}\sum_{i=1}^n {X}_i\right)$ and variance equal to $\mathrm{Var}\left(n^{-1/2}\sum_{i=1}^n {X}_i\right)$. For notational simplicity, set $S_{n,{X}} = n^{-1/2}\sum_{i=1}^n {X}_i$.

Note that, if ${X_1},\cdots,{X_n}$ are i.i.d distributed with mean ${0}$ and variance ${\Sigma}$, then ${Z}\sim\gauss({0},{\Sigma})$. If ${X_1},\cdots,{X_n}$ are independent with mean ${0}$, then ${Z}\sim\gauss\left({0},n^{-1}\sum_{i=1}^{n}\mathrm{Var}({X_i})\right)$. 

Suppose $f$ is a Borel measurable function. We wish to bound
\begin{equation}
\Delta_f = \abs{\mathbb{E}f\left(S_{n,{X}}\right) - \mathbb{E}f\left({Z}\right)},
\end{equation} 
with a constant which depends on $n$, $d$ and the fucntion $f$. Classical Berry--Esseen bounds~\citep{bentkus2005lyapunov,raivc2019multivariate} provide inequalities for $\Delta_f$ when $f(x) = \mathbf{1}{x\in A}, x\in\mathbb{R}^d$ for some sets $A$ (e.g., convex sets, Euclidean balls). Results that bound $\Delta_f$ for Borel measurable functions with specific polynomial growth are also well-known; see~\citet[Theorem 1, Chapter 1, Section 3]{Sazonov},~\citet[Thms 13.2, 13.3]{bhattacharya2010normal}, and~\citet[Theorem 4.1]{angst2017weak}. These bounds are derived using smoothing inequalities and are valid for all Borel measurable functions. Unfortunately, these bounds are not sharp enough to imply correct dimension on dimension or smoothness for high-dimensional functions that are ``highly'' smooth (e.g., high-dimensional functions that depend only on a subset of coordinates); see, for example,~\citet[Thms 3.2--3.4]{bentkus2003new} for sharp bounds on $\Delta_f$ for functions $f$ in H{\"o}lder classes. Further, the dependence on dimension is much better than that implied by the general result in~\cite{Sazonov} and~\cite{bhattacharya2010normal}.

In this paper, we propose an approximation theory and level sets based approach to obtain bounds for $\Delta_f$. The bounds we obtain have sharp dependence on the dimension as well as the sample size even when the dimension grows faster than the sample size. On the flip side, our bounds do not apply to all Borel measurable functions but only to a special class of functions.

The remaining article is organized as follows. In Section~\ref{sec:literature-review}, we provide a discussion of existing Berry--Esseen bounds for independent random vectors from~\cite{bentkus2005lyapunov} and~\cite{raivc2019multivariate}. These results will be the backbone of our approach. In Section~\ref{sec:bound_when_the_conditions_of_Bentkus_hold_true}, we provide our first result that bounds $\Delta_f$, for a bounded $f$, in terms of the upper/lower level sets of $f$. In Section~\ref{sec:applications_thm}, we provide an application of this result to bound $\Delta_f$ for bounded quasi-concave functions. In Section~\ref{sec:non_uniform_BE_bnd}, we provide a discussion of non-uniform Berry--Esseen bound and apply our main result to get bounds for functions in Barron space. We conclude the article with a brief discussion in Section~\ref{sec:discussion}.

\section{Literature Survey}\label{sec:literature-review}

To begin with, we take a look at Theorem 1.2 of \cite{bentkus2005lyapunov}. Firstly, we define some notations as per \cite{bentkus2005lyapunov}.
Let $X_1,\cdots,X_n$ be independent random vectors with a common mean $\E X_j = 0$. We write, $S = X_1 + \cdots + X_n$. Throughout we assume that S has a non-degenerated distribution in the sense that the covariance operator, say $C^2 = \cov(S)$, is invertible (where, $C$ stands for the positive root of $C^2$). Let $Z$
be a Gaussian random vector such that $\E Z = 0$ and $\cov(S)$ and $\cov(Z)$ are equal. We further write
\[\beta = \beta_1 + \cdots + \beta_n, \quad \beta_k = \E|C^{-1}X_k|^3.\]
and for any collection of sets $\mathcal{A}$, define
\[\Delta(\mathcal{A}) = \sup_{A \in \mathcal{A}}|\mathbb{P}\{S \in A\} - \mathbb{P}\{Z \in A\}|.\]

\begin{theorem}
\label{theorem 1.2 of Bentkus 2005}

Let a class $\mathcal{A}$ of convex sets  of subsets $A \subset \mathbb{R}^d$ satisfy the following conditions. 

\begin{itemize}
    \item[(i)] Class $\mathcal{A}$ is invariant under affine symmetric transformations, i.e., $DA + a \in A$ if $a \in \mathbb{R}^d$ and $D : \mathbb{R}^d \longrightarrow \mathbb{R}^d$ is a linear symmetric invertible operator. 
    \item[(ii)] Class $\mathcal{A}$ is invariant under taking $\varepsilon$-neighbourhoods for all $\varepsilon > 0$. More precisely, $A^\varepsilon, A^{-\varepsilon} \in \mathcal{A}$, if $A \in \mathcal{A}$. Here, 
    \[A^\varepsilon = \{x \in \R^d : \rho_A(x) \leq \varepsilon\}  \quad \text{and} \quad A^{-\varepsilon} = \{x \in \R^d : B_\varepsilon(x) \subset A\}\]
    where, $\rho_A(x) = \inf_{y \in A} |x-y|$ is the distance between $A \subset \R^d$ and $x \in \R^d$ and $B_\varepsilon(x) = \{y \in \R^d : |x-y| \leq \varepsilon\}$.
\end{itemize}
    
    Let $\phi$ denote the standard normal distribution.
    Furthermore, assume that, $\mathcal{A}$ and the standard normal distribution satisfy the condition that, there exists constants, say $a_d = a_d(\mathcal{A})$, called the isoperimetric constant of $\mathcal{A}$, depending only on $d$ and $\mathcal{A}$, such that, 
    \[\phi\{A^\varepsilon \setminus A\} \leq a_d\varepsilon \quad \text{and} \quad \phi\{A \setminus A^{-\varepsilon}\}  \leq a_d\varepsilon \quad \text{for all} \quad A \in \mathcal{A} \quad \text{and} \quad \varepsilon > 0.\]
    
If these two conditions and the assumption hold, then there exists
an absolute constant $M > 0$ such that
\[\Delta(\mathcal{A}) \leq Mb_d\beta, \quad b_d = \max\{1,a_d(\mathcal{A})\}.\]
\end{theorem}

\textbf{Note}: Conditions (i)-(ii) on the class $\mathcal{A}$ can be relaxed using slightly more refined techniques. In the i.i.d. case one can relax requirement (i) on $\mathcal{A}$, assuming that $\mathcal{A}$ is invariant under rescaling by scalars and shifting. 

\paragraph{Example bounds on isoperimetric constant.}
\begin{enumerate}
    \item $a_d(\A) = (2\pi)^{-1/2}$ in the case of the class $\mathcal{H}$ of all affine half-spaces of $\R^d$~\citep{bentkus2003dependence}.
    
    \begin{definition}
    A class of half-spaces $\A^\text{half space}$ is defined as $\A^\text{half space} = \{A_{a,b}:a\in\R^d,b\in\R\}$, where, $A_{a,b} = \{x : a^{\top}x \leq b\}$.
    \end{definition}
    
    \item For the class $\mathcal{C}$ of all convex subsets of $\R^d$, we have $a_d(\mathcal{C}) \leq 4d^{1/4}$~\citep[Lemma 2.6]{bentkus2003dependence}. 

    \item For the class $\mathcal{B}$ of all Euclidean balls of $\mathbb{R}^d$, we have $a_d(\mathcal{B}) \leq C$ for some absolute constant $C$~\citep[Lemma 8.1]{zhilova2020new}.
    
\end{enumerate}

Now, we take a look at a more general form of \cref{theorem 1.2 of Bentkus 2005} as given by \cite{raivc2019multivariate}. We define some notations as per \cite{raivc2019multivariate}.
Let $\mathcal{I}$ be a countable set (either finite or infinite) and let $\{X_i\}_{i\in \mathcal{I}}$, be independent
$\R^d$-valued random vectors. Assume that $\E X_i = 0$ for all $i$ and that $\sum_{i \in \mathcal{I}} \text{Var}(X_i) = {I}_d$.
It is well known that in this case, the sum $W := \sum_{i \in \mathcal{I}}X_i$ exists almost surely and that $\E W = 0$ and $\text{Var}(W) = {I}_d$.

For a measurable set $A \subset \R^d$, let $\gauss(\mu,\Sigma)\{A\} := \mathbb{P}(Z \in A)$ and for a measurable function $f : \R^d \longrightarrow \R$, let $\gauss(\mu,\Sigma)\{f\} := \E f(Z)$, where $Z \sim \gauss(\mu,\Sigma)$.

We shall consider a class $\mathcal{A}$ of measurable sets in $\R^d$. For each $A \in \mathcal{A}$ , we take a measurable function $\rho_A : \R^d \longrightarrow \R$. The latter can be considered as a generalized signed distance function. Typically, one can take $\rho_A = \delta_A$, where, 
\begin{equation*}
\delta_A(x) :=
\begin{cases}
    -\text{dist}(x,\R^d\setminus A), &\quad x \in A \\
    \text{dist}(x,A), &\quad x \notin A.
\end{cases} 
\end{equation*}
But we allow for more general functions. For each $t \in \R$, we define \[A^{t|\rho} = \{x : \rho_A(x) \leq t\}.\]

We define the generalised Gaussian perimeter as:

\[\gamma^*(A \setminus \rho) := \text{sup} \bigg\{\frac{1}{\varepsilon}\gauss(0,{I}_d)\{A^{\varepsilon|\rho}\setminus A\},\{\frac{1}{\varepsilon}\gauss(0,{I}_d)\{A\setminus A^{-\varepsilon|\rho}\};\varepsilon>0\bigg\}.\]

\[\gamma^*(\A \setminus \rho) := \sup\gamma^*(A \setminus \rho).\]

Now, we consider the following assumptions:

\begin{itemize}
    \item[(A1)] $\mathcal{A}$ is closed under translations and uniform scalings by factors greater than one.
    \item[(A2)] For each $A \in \mathcal{A}$ and $t \in \R$, $A^{t|\rho} \in \mathcal{A} \cup \{\phi,\R^d\}$.
    \item[(A3)] For each $A \in \mathcal{A}$ and $\varepsilon > 0$, either $A^{-\varepsilon|\rho} = \phi$ or $\{x : \rho_{A^{-\varepsilon|\rho}}(x) < \varepsilon \subseteq A\}$.
    \item[(A4)] For each $A \in \mathcal{A}$, $\rho_A(x) \leq 0$ for all $x \in A$ and $\rho_A(x) \geq 0$ for all $x \notin A$.
    \item[(A5)] For each $A \in \mathcal{A}$ and each $y \in \R^d$, $\rho_{A+y}(x + y) = \rho_A(x)$ for all $x \in \R^d$.
    \item[(A6)] For each $A \in \mathcal{A}$ and each $q \geq 1$, $|\rho_{qA}(qx)| \leq q|\rho_{A}(x)|$ for all $x \in \R^d$.
    \item[(A7)] For each $A \in \mathcal{A}$, $\rho_A$ is non-expensive on $\{x : \rho_A(x) \geq 0\}$, i.e., $|\rho_A(x) - \rho_A(y)| \leq |x - y|$ for all $x,y$ with $\rho_A(x) \geq 0$ and $\rho_A(y) \geq 0$.
    \item[(A8)] For each $A \in \mathcal{A}$, $\rho_A$ is differentiable on $\{x : \rho_A(x) > 0\}$. Moreover, there exists $\kappa \geq 0$, such that \[|\nabla\rho_A(x) - \nabla\rho_A(y)| \leq \frac{\kappa|x - y|}{\text{min}\{\rho_A(x),\rho_A(y)\}}\] for all $x,y$ with $\rho_A(x) > 0$ and $\rho_A(y) > 0$, where $\nabla$ denotes the gradient.
\end{itemize}

We also consider another optional assumption as follows:

\begin{itemize}
    \item[(A1')] $\mathcal{A}$ is closed under symmetric linear transformations with the smallest eigenvalue at least one.
\end{itemize}

\begin{theorem}

\label{theorem 1.3 of Raic 2019}
Let $W = \sum_{i \in \mathcal{I}}X_i$ and let $\mathcal{A}$ be a class of sets meeting assumptions (A1)–(A8) (along with the underlying functions $\rho_A$). Then for each
$A \in \mathcal{A}$ , the following estimate holds true 

\[|\mathbb{P}(W \in A) - \gauss(0,{I}_d)\{A\}| \leq \max\{27,1+53\gamma^*(\mathcal{A}|\rho)\sqrt{1+\kappa}\}\sum_{i \in \mathcal{I}}\E|X_i|^3. \]

In addition, if $\mathcal{A}$ also satisfies assumption (A1'), then the preceding bound can be improved to

\[|\mathbb{P}(W \in A) - \gauss(0,{I}_d)\{A\}| \leq \max\{27,1+50\gamma^*(\mathcal{A}|\rho)\sqrt{1+\kappa}\}\sum_{i \in \mathcal{I}}\E|X_i|^3 .\]

\end{theorem}

\subsubsection{Examples of Classes of Sets satisfying (A1)-(A8) of \cref{theorem 1.3 of Raic 2019}}

\begin{itemize}

    \item[(i)] the class $\mathcal{C}_d$ of all measurable convex sets in $\R^d$, along with $\rho_A = \delta_A$, which is defined in $\mathcal{C}_d\setminus\{\phi,\R^d\}$.
    
     \item[(ii)] the class of all balls in $\R^d$ (excluding the empty set) along with $\rho_A = \delta_A$ (since the balls are convex, it meets all the assumptions (A1)-(A8)).
     
     \item[(iii)] For a class of ellipsoids, $\rho_A = \delta_A$  is not suitable because an $\varepsilon$-neighborhood of an ellipsoid is not an ellipsoid. However, one can set $\rho_A(x) = \delta_{QA}(Qx)$ , where $Q$ is a linear transformation mapping $A$ into a ball (may depend on $A$). [Note that, $Q$ must be non-expansive in order to satisfy (A7).]

\end{itemize}




\section{Level sets and bounds on $\Delta_f$}
\label{sec:bound_when_the_conditions_of_Bentkus_hold_true}

If the conditions of \Cref{theorem 1.2 of Bentkus 2005} are met, then we have

\[\sup_{A\in\A}\abs{\P\left(S_{n,{X}}\in A\right) - \P\left({Z}\in A\right)} \leq \frac{b_d(\A)}{\sqrt{n}}\frac{1}{n}\sum_{i=1}^n\E|\Sigma^{-\frac{1}{2}}X_i|^3\]
or in other words, we have, 
\[\sup_{A\in\A}\abs{\mathbb{E}f\left(S_{n,{X}}\right) - \mathbb{E}f\left({Z}\right)} \leq \frac{b_d(\A)}{\sqrt{n}}\frac{1}{n}\sum_{i=1}^n\E|\Sigma^{-\frac{1}{2}}X_i|^3\]
where, $f$ is defined as $f(x) = \mathbb{I}(x \in A)$, $\Sigma = \mathrm{Var}\left(S_{n,{X}}\right)$ and $b_d(\A)$ is as defined in \cref{theorem 1.2 of Bentkus 2005}.

\begin{theorem}
\label{lemma 4.1}
Let $f$ be a bounded Borel measurable function. For each $t\in\mathbb{R}$, define the upper level sets of $f(\cdot)$ at level $t$ as $A_t = \{x\in\mathbb{R}^d:\,f(x) \ge t\}$. Then, we have,
\begin{equation*}
    \abs{\mathbb{E}f\left(S_{n,{X}}\right) - \mathbb{E}f\left({Z}\right)} \leq 2 \norm{f}_{\infty}\sup_{t\in\mathbb{R}}\left|\mathbb{P}\left(S_{n,{X}}  \in A_t\right) - \mathbb{P}(Z \in A_t)\right|.
\end{equation*}
\end{theorem}
\Cref{lemma 4.1} implies that for bounded functions $f$, $\Delta_f$ can be controlled using classical Berry--Esseen bounds if the level sets belong to a ``favorable'' class.
The scope of~\Cref{lemma 4.1} can be expanded significantly using the following two combinations.
\subsection{Combination 1}
Suppose $f_1,\cdots,f_k$ are bounded functions (satisfying the conditions of \cref{theorem 1.2 of Bentkus 2005}) for which the upper level sets all belong to $\A$, where $\A$ is any one of the favorable classes of convex sets or half-spaces or Euclidean balls. Then, for a function $f$ such that $f(x) = \sum_{j=1}^k \lambda_j f_j(x)$, where $\sum_{j=1}^k \abs{\lambda_j} \leq c$ for some constant $c$, by \Cref{lemma 4.1}, we have, 
\begin{equation*}
\begin{split}
   \abs{\mathbb{E}f\left(S_{n,{X}}\right) - \mathbb{E}f\left({Z}\right)} & = \abs{\sum_{j=1}^k \lambda_j \bigg\{\mathbb{E}f_j\left(S_{n,{X}}\right) - \mathbb{E}f_j\left({Z}\right)\bigg\}} \\ 
   & \leq \sum_{j=1}^k\abs{\lambda_j \bigg\{\mathbb{E}f_j\left(S_{n,{X}}\right) - \mathbb{E}f_j\left({Z}\right)\bigg\}} \\ 
   & \leq \sum_{j=1}^k\Bigg[\abs{\lambda_j}\abs{ \bigg\{\mathbb{E}f_j\left(S_{n,{X}}\right) - \mathbb{E}f_j\left({Z}\right)\Bigg\}}\bigg] \\
    & \leq \sup_j\abs{ \bigg\{\mathbb{E}f_j\left(S_{n,{X}}\right) - \mathbb{E}f_j\left({Z}\right)\bigg\}} \sum_{j=1}^k \abs{\lambda_j} \\
   & \leq c \sup_j\abs{ \bigg\{\mathbb{E}f_j\left(S_{n,{X}}\right) - \mathbb{E}f_j\left({Z}\right)\bigg\}} \\
   & \leq c . \sup_j \Bigg[2\norm{f_j}_{\infty} \sup_t \abs{\P\left(S_{n,{X}} \in A_{j,t}  \right) - \P\left(Z\in A_{j,t}\right)}\Bigg]\\
   &\qquad [\text{we got the above step using~\Cref{lemma 4.1}}] \\
   &  = 2c.\sup_j\Bigg[\norm{f_j}_{\infty} \sup_t \abs{\P\left(  \in A_{j,t}  \right) - \P\left(Z\in A_{j,t}\right)}\Bigg]\\
   & \leq 2c.\sup_j\Bigg[\norm{f_j}_{\infty}  \frac{b_d(\A)}{\sqrt{n}}\frac{1}{n} \left(\sum_{i=1}^n\E|\Sigma^{-\frac{1}{2}}X_i|^3\right)\Bigg] \\
   &\qquad [\because \text{ each } f_j \text{ has its upper level sets in the same favorable class } \A] \\
   & = \frac{2b_d(\A)c}{\sqrt{n}}  \left(\frac{1}{n}\sum_{i=1}^n\E|\Sigma^{-\frac{1}{2}}X_i|^3\right)\sup_j\norm{f_j}_{\infty}.
\end{split}
\end{equation*}
A direct application of this result can yield bounds for functions that are uniformly approximable by a class of functions whose level sets belong in the favorable class. This will be further explored in Part II.
The following result shows that one can construct an infinite class of functions with the same upper-level sets from a given function. 
\begin{theorem}
If the upper-level sets of a function $f: \R^n \longrightarrow \R$ belong in $\A$, then the upper-level sets of $g\circ f$ also belong in $\A$ for any non-decreasing function $g: \R \longrightarrow \R$. 
\end{theorem}

\begin{proof}
Let $A'_t$ be the upper level sets of $g$ for $t \in \R$, Then, we have, $A'_t = \{x : g(f(x)) \geq t\}$ . 
Now, since $g$ is a non-decreasing function, we have, $g(f(x)) \geq t \iff f(x) \geq g^{-1}(t)$. 
Thus, we have $A'_t = \{x : g(f(x)) \geq t\} = \{x : f(x) \geq g^{-1}(t)\} = A_{g^{-1}(t)}$, where, $A_t = \{x : f(x) \geq t\}$ denotes the upper level set of $f$ for $t \in \R$.
We know that $A_g^{-1}(t)$ belongs in $\A$ and hence $A'_t$ also belongs in $\A$.
\end{proof}

\begin{corollary}
If the upper level sets of a bounded function $f: \R^n \longrightarrow \R$ belong in $\A$, then the upper level sets of $h$ such that $h(x) = \frac{f(x)}{\norm{f}_\infty}$ also  belong in $\A$.
\end{corollary}

\subsection{Combination 2}
Suppose a function $f: \R^n \longrightarrow \R$ has an upper level set $U_t$ of the form $U_t = A_{1,t} \cup \cdots \cup A_{k,t}$, where $A_{j,t}$'s are disjoint sets each belonging to a favorable class, $\forall t \in \R$, then also we can bound $\Delta_f$.
If $f$ is a bounded function, then by~\Cref{lemma 4.1}, we have, 
\begin{equation*}
\begin{split}
    \abs{\mathbb{E}f\left(S_{n,{X}}\right) - \mathbb{E}f\left({Z}\right)} & \leq 2\norm{f}_{\infty} \sup_t \abs{\P\left(S_{n,{X}}  \in U_t \right) - \P\left(Z\in U_t\right)} \\
    & = 2\norm{f}_{\infty} \sup_t \abs{\P\left(S_{n,{X}}  \in \cup_{j=1}^{k}{A_{j,t}}  \right) - \P\left(Z\in \cup_{j=1}^{k}{A_{j,t}}\right)}\\
    & = 2\norm{f}_{\infty} \sup_t \abs{\sum_{j=1}^k\Bigg[\P\left(S_{n,{X}}  \in  A_{j,t}  \right) - \P\left(Z\in  A_{j,t}\right)\Bigg]} \\
    &\qquad [\because A_{j,t}\text{'s are disjoint}] \\
    & \leq 2\norm{f}_{\infty} \sup_t \sum_{j=1}^k \abs{\Bigg[\P\left(S_{n,{X}}  \in  A_{j,t}  \right) - \P\left(Z\in  A_{j,t}\right)\Bigg]} \\
    & \leq 2\norm{f}_{\infty}\sum_{j=1}^k \sup_t \abs{\Bigg[\P\left(S_{n,{X}}  \in  A_{j,t}  \right) - \P\left(Z\in  A_{j,t}\right)\Bigg]} \\
    & \leq 2\norm{f}_{\infty}\sum_{j=1}^k\frac{b_d(\A_j)}{\sqrt{n}}\frac{1}{n}\left(\sum_{i=1}^n\E|\Sigma^{-\frac{1}{2}}X_i|^3\right) \\
    & = \frac{2\norm{f}_{\infty}}{\sqrt{n}}\sum_{j=1}^k b_d(\A_j) \left(\frac{1}{n}\sum_{i=1}^n\E|\Sigma^{-\frac{1}{2}}X_i|^3\right)
\end{split}
\end{equation*}
\subsection{Proof of Theorem~\ref{lemma 4.1}}
\begin{proof}

To begin with, we consider a non-negative function $f$. Then, we can write, 
\begin{equation*}
\begin{split}
f(x) & = \int_{0}^{\infty} \I(f(x)\geq t)dt 
 = \int_{0}^{f(x)}dt + \int_{f(x)}^{\infty}0dt 
 = f(x)
\end{split}
\end{equation*}
Now, $\abs{f(x)} \leq \norm{f}_{\infty}$. Thus, we can now write, 
\begin{equation*}
\begin{split}
& f(x) = \int_{0}^{\norm{f}_{\infty}} \I(f(x)\geq t)dt \\
\implies & \E f(W) = \E \int_{0}^{\norm{f}_{\infty}} \I(f(W)\geq t)dt
\end{split}
\end{equation*}
By Tonelli's theorem, we know that the swapping of the expectation operator and integral is valid for non-negative summands. For the convenience of the reader, we recall Tonelli's theorem here: Suppose that  $(T,\mathcal{T},\mu)$  is a  $\sigma$-finite measure space, and that $X_t$  is a real-valued random variable for each  $t\in T$. We assume that  $(\omega,t) \mapsto X_t(\omega)$  is measurable, as a function from the product space $(\Omega \times T, \mathcal{F} \otimes \mathcal{T})$  into  $\R$. If $X_t$ is non-negative for each $t \in \R$, then, 
\begin{equation*}
    \E \int_{T} X_t d\mu(t) = \int_{T} \E (X_t) d\mu(t)
\end{equation*}
Thus, we have,   
\[\E f(W) =  \int_{0}^{\norm{f}_{\infty}}\E \I(W \geq t)dt = \int_{0}^{\norm{f}_{\infty}} \P(f(W)\geq t)dt\]
Now, we consider the expression $\abs{\mathbb{E}f\left(S_{n,{X}}\right) - \mathbb{E}f\left({Z}\right)}$. 
\begin{equation*}
\begin{split}
    \abs{\mathbb{E}f\left(S_{n,{X}}\right) - \mathbb{E}f\left({Z}\right)} & = \abs{\int_{0}^{\norm{f}_{\infty}}     \left[\P\left(S_{n,{X}}  \geq t \right) - \P\left(Z\geq t\right)\right] dt }\\
    & \leq \int_{0}^{\norm{f}_{\infty}} \abs{\P\left(S_{n,{X}}  \geq t \right) - \P\left(Z\geq t\right)} dt\\
    & = \int_{0}^{\norm{f}_{\infty}}     \abs{\P\left(S_{n,{X}}  \in A_t \right) - \P\left(Z\in A_t\right)} dt \\
     & \leq \norm{f}_{\infty} \sup_{t \in \R} \abs{\P\left(S_{n,{X}}  \in A_t \right) - \P\left(Z \in A_t\right)}
\end{split}
\end{equation*}
Now, suppose that $f$ is not non-negative, but bounded. Then, we have, 
\begin{equation*}
\begin{split}
     \mathbb{E}f\left(S_{n,{X}}\right) - \mathbb{E}f\left({Z}\right) & =  \E\Bigg[ \int_{0}^{\norm{f}_{\infty}} \I(f\left(S_{n,{X}}\right)\geq t)dt - \int_{-\norm{f}_{\infty}}^{0} \I(f\left(S_{n,{X}}\right)\leq t)dt  \Bigg] \\
     &\quad- \E\Bigg[ \int_{0}^{\norm{f}_{\infty}} \I(f(Z)\geq t)dt - \int_{-\norm{f}_{\infty}}^{0} \I(f(Z)\leq t)dt  \Bigg] 
  \end{split}   
\end{equation*}
By Tonelli's theorem, we know that the swapping of the expectation operator and integral is valid for non-negative summands. Thus, we have,  
\begin{equation*}
    \begin{split}
     \mathbb{E}f\left(S_{n,{X}}\right) - \mathbb{E}f\left({Z}\right)  & = \Bigg[\int_{0}^{\norm{f}_{\infty}} \E\I(f\left(S_{n,{X}}\right)\geq t)dt - \int_{-\norm{f}_{\infty}}^{0} \E\I(f\left(S_{n,{X}}\right)\leq t)dt\Bigg]   \\
     &\quad - \Bigg[\int_{0}^{\norm{f}_{\infty}} \E\I(f(Z)\geq t)dt - \int_{-\norm{f}_{\infty}}^{0} \E\I(f(Z)\leq t)dt\Bigg]    \\
     & = \int_{0}^{\norm{f}_{\infty}}     \left[\P\left(S_{n,{X}}  \in A_t \right) - \P\left(Z\in A_t\right)\right] dt\\
     &\quad - \int_{-\norm{f}_{\infty}}^{0}    \left[\P\left(S_{n,{X}}  \in B_t \right) - \P\left(Z\in B_t\right)\right] dt \\
    &\quad \text{where, }  B_t \text{is the lower level sets of }  f \text{ defined as } B_t = \{x \in \R^d : f(x) \leq t\}
    \end{split}
\end{equation*}

In the case of general functions $f$, which may not be non-negative, since the integral doesn't run only from $0$ to $\norm{f}_\infty$, there is an occurrence of the lower level sets $ B_t$'s. But since the probability of lying in $B_t$'s can be expressed in terms of lying in $A_t$'s, we can easily write the whole bound in terms of $A_t$'s.

\begin{equation*}
    \begin{split}
    \therefore \abs{\E f\left(S_{n,{X}}\right) - \E f\left({Z}\right)} & = \abs{\int_{0}^{\norm{f}_{\infty}}     \left[\P\left(S_{n,{X}}  \in A_t \right) - \P\left(Z\in A_t\right)\right] dt -  \int_{-\norm{f}_{\infty}}^{0}    \left[\P\left(S_{n,{X}}  \in B_t \right) - \P\left(Z\in B_t\right)\right] dt}\\
    & = \abs{\int_{0}^{\norm{f}_{\infty}}     \left[\P\left(S_{n,{X}}  \in A_t \right) - \P\left(Z\in A_t\right)\right] dt -  \int_{-\norm{f}_{\infty}}^{0}    \left[\P\left(Z\in A_t\right) - \P\left(S_{n,{X}}  \in A_t \right) \right] dt} \\
    & [\because \P(W \in B_t) = 1 - \P(W \in A_t)] \\
    & = \abs{\int_{0}^{\norm{f}_{\infty}}     \left[\P\left(S_{n,{X}}  \in A_t \right) - \P\left(Z\in A_t\right)\right] dt +  \int_{-\norm{f}_{\infty}}^{0}    \left[\P\left(S_{n,{X}}  \in A_t \right) - \P\left(Z\in A_t\right)  \right] dt}  \\
    & = \abs{\int_{-\norm{f}_{\infty}}^{\norm{f}_{\infty}}     \left[\P\left(S_{n,{X}}  \in A_t \right) - \P\left(Z\in A_t\right)\right] dt} \\
    & \leq \int_{-\norm{f}_{\infty}}^{\norm{f}_{\infty}} \abs{\P\left(S_{n,{X}}  \in A_t \right) - \P\left(Z\in A_t\right)} dt \\
    & \leq 2 \norm{f}_{\infty}\sup_{t\in\mathbb{R}}\abs{\P\left(S_{n,{X}}  \in A_t\right) - \P(Z \in A_t)}
    \end{split}
\end{equation*}
\end{proof}


\section{Applications of~\Cref{lemma 4.1}}\label{sec:applications_thm}

\subsection{Classes of functions whose upper-level sets are convex sets}

\begin{definition}
Suppose $X$ is a convex subset of $\R^d$. A function $f : X \longrightarrow R$ is \textit{quasiconcave} if
\begin{equation}\label{eq:quasi-concave-definition}
f(\alpha x + (1-\alpha) y) \geq \min\{f(x), f(y)\}  \quad \forall \alpha \in [0,1], x, y\in X.
\end{equation}
\end{definition}
It is well-known that quasiconcave functions are precisely those functions whose upper-level sets are convex. It is reproduced below for reader's convenience. 
\begin{theorem}\label{thm:quasi-concave}
A function $f : X \longrightarrow R$ is quasiconcave if and only if the upper level sets of $f$, i.e., the sets $A_t = \{x:f(x)\geq t\}$ are convex for all $t \in \R$.
\end{theorem}

\begin{proof}
Suppose that $f : X \longrightarrow R$ is quasiconcave.

Fix a $t\in\mathbb{R}$. Consider the points $x,y \in A_t = \{z:f(z)\geq t\}$. Then, $f(x) \geq t$ and $f(y) \geq t$. Then by~\eqref{eq:quasi-concave-definition}, $\forall \alpha \in [0,1]$, we have,
\begin{equation*}
       f(\alpha x + (1-\alpha)y) \geq \min\{f(x),f(y)\}  \ge t.
\end{equation*}
This means that $\alpha x + (1-\alpha)y \in A_t$, thus proving that $A_t$ is convex.
Now, suppose that the upper level sets of $f$, $A_t = \{z:f(z)\geq t\}$ is convex $\forall t \in \R$. Consider two points $x,y$. Let $t' = \min\{f(x),f(y)\}$. Then, 
\begin{equation*}
    x \in A_{t'}  \text{ and } y \in A_{t'}.
\end{equation*}
Now, $A_{t'}$ is convex $\implies \alpha x + (1-\alpha)y \in A_{t'} \forall \alpha \in [0,1].$
Then, from the definition of $A_{t'}$ we have, 
\begin{equation*}
    f(\alpha x + (1-\alpha)y) \geq t' = \min\{f(x),f(y)\} 
\end{equation*}
Since, $x,y$ are arbitrary, we have proved that $f$ is quasiconcave. 
\end{proof}

\Cref{thm:quasi-concave} implies that $\Delta_f$ for quasi-concave functions bounded by $M$ can be bounded by $2M\Delta(\mathcal{C})$. 
The class of functions whose upper-level sets are half-spaces are ridge functions of the form $f(x) = \sigma(a^{\top}x - b)$ for some monotone function $\sigma(\cdot)$. These are discussed in later sections. The class of functions whose upper-level sets are Euclidean balls are radial functions, where the upper-level sets belong to the favorable class $\mathcal{B}$; these will be discussed in part II.

Until now, we have focused on only bounded functions. In the following section, we first discuss non-uniform Berry--Esseen bound and provide applications for unbounded functions whose level sets are half-spaces.
\section{Non-uniform Berry--Esseen bound and half-space level sets}\label{sec:non_uniform_BE_bnd}
To begin with, we take a look at the following bound from  \cite{shevtsova2020lower}. 
\begin{theorem}
\label{4 of Shevtsova}
Let $X_1,X_2,\cdots,X_n$ be independent random variables with distribution functions $F_1, F_2,\cdots,F_n$ and $\E X_k = 0$, $\sigma^2_k := \E X_k < \infty$,  
\begin{equation*}
    S_n := \sum_{k=1}^n X_k, B_n^2 = \sum_{k=1}^n \sigma_k^2 > 0. 
\end{equation*}
Let us denote 
\begin{equation*}
    \bar{F}_n(x) := P(S_n < xB_n), \quad \Phi(x) = \int_{-\infty}^x e^{-\frac{t^2}{2}} dt, \quad \Delta_n(x) = \abs{\bar{F}_n(x) - \Phi(x)}, \quad x \in \R. 
\end{equation*}
\begin{equation*}
    L_n(\varepsilon) := \frac{1}{B_n^2}\sum_{k=1}^n \E X_k^2 \I(\abs{X_k} > \varepsilon B_n), \quad \varepsilon > 0, \quad n \in \N. 
\end{equation*}
Then, 
\begin{equation}
    \begin{split}
        \Delta_n(x) & \le \frac{A}{(1 + \abs{x}^3)B_n}\int_{0}^{(1 + \abs{x})B_n} L_n(z) dz \\
        & = A\sum_{k=1}^n\Bigg[\frac{E X_k^2 \I(\abs{X_k} > (1+\abs{x})B_n)}{(1+\abs{x})^2B_n^2} + \frac{E X_k^2 \I(\abs{X_k} \le (1+\abs{x})B_n)}{(1+\abs{x})^3B_n^3} \Bigg]
    \end{split}
\end{equation}
where $A$ is an absolute constant. 
\end{theorem}






\subsection{Bounds for the ReLU and squared ReLU functions} 
\begin{theorem}
\label{bound for relu function}
Let $\sigma$ denote the ReLU function, i.e., $\sigma(x) = \max\{0,x\}$. Let $W_1, W_2, \dots , W_n$ be independent univariate random variables with $\E W_k = 0$ and $\sigma_k^2 = \E W_k^2 < \infty$, $\forall k = 1, 2, \dots, n$. Let us denote 
\begin{equation*}
    S_{n,W}\coloneqq \sum_{k=1}^n W_k, \quad B_{n,W}^2 \coloneqq \sum_{k=1}^n \sigma_k^2.
\end{equation*}
Let $Z$ be a standard normal random variable. Then, for $t\geq 0$, 
\begin{equation*}
\begin{split}
     \abs{\E\sigma(S_{n,W}-t) - \E\sigma(B_{n,W}Z-t)}   & \leq \frac{A}{B_{n,W}}\sum_{k=1}^n \E W_k^2\I(\abs{W_k} \geq t + B_{n,W})  +
     \frac{A}{2}\sum_{k=1}^n \frac{\E \abs{W_k}^3\I(\abs{W_k} < t+B_{n,W})}{(t+B_{n,W})^2} .
\end{split}
\end{equation*}
for an absolute constant $A$.
\end{theorem}

\begin{proof}

For $t \geq 0$,

\begin{equation*}
    \begin{split}
    &\abs{\E\sigma\left(\frac{S_{n,W}}{B_{n,W}} - \frac{t}{B_{n,W}}\right) - \E\sigma\left(Z- \frac{t}{B_{n,W}}\right)}\\ 
    & = \abs{\int_{0}^{\infty}\left[\P\left(\sigma\left(\frac{S_{n,W}}{B_{n,W}} - \frac{t}{B_{n,W}}\right) \geq w \right) - \P\left(\sigma\left(Z- \frac{t}{B_{n,W}}\right) \geq w \right) \right]dw} \\
     & = \abs{\int_{0}^{\infty}\left[\P\left(\sigma\left(Z- \frac{t}{B_{n,W}}\right) < w \right) - \P\left(\sigma\left(\frac{S_{n,W}}{B_{n,W}} - \frac{t}{B_{n,W}}\right) < w \right) \right]dw} \\
     & \leq  \int_{0}^{\infty}\abs{\P\left(\sigma\left(Z- \frac{t}{B_{n,W}}\right) < w \right) - \P\left(\sigma\left(\frac{S_{n,W}}{B_{n,W}} - \frac{t}{B_{n,W}}\right) < w \right)} dw  \\
     & = \int_{0}^{\infty}\Bigg |[\P\left(\max\Bigg\{0,Z-\frac{t}{B_{n,W}}\Bigg\} < w \right) -  \P\left(\max\Bigg\{0,\frac{S_{n,W}}{B_{n,W}} - \frac{t}{B_{n,W}}\Bigg\} < w \right) \Bigg|  dw.
    \end{split}
\end{equation*}
(Since we are considering the integral from  $0$ to $\infty$,  so we have $w > 0$.) So, for a random variable $X, w > \max\{0,X\}$ means that $w > 0$ and $w > X$. But because of the lower limit of the integral, we already have $w > 0$. So, the event $w > 0$ and $w > X$ boils down to the event $w > X$.
\begin{equation}
    \begin{split}
        &\abs{\E\sigma\left(\frac{S_{n,W}}{B_{n,W}} - \frac{t}{B_{n,W}}\right) - \E\sigma\left(Z- \frac{t}{B_{n,W}}\right)}\\ 
        & \leq \int_{0}^{\infty}\abs{\left[\P\left(Z - \frac{t}{B_{n,W}} < w \right) - \P\left(\frac{S_n}{B_n} - \frac{t}{B_{n,W}} < w \right) \right]dw} \\
     & = \int_{0}^{\infty}\abs{\left[\P\left(Z < \frac{t}{B_{n,W}} + w \right) - \P\left(\frac{S_n}{B_n} < \frac{t}{B_{n,W}} +  w \right) \right]dw} \\
     & = \int_{0}^{\infty} \Delta_n\left(w + \frac{t}{B_{n,W}}\right) dw \\
     & = \int_{{t}/{B_{n,W}}}^{\infty} \Delta_n (s) ds \quad \Bigg[\text{by the change of variable } s = w + \frac{t}{B_{n,W}}\Bigg] 
    \end{split}
\end{equation} 
By~\Cref{4 of Shevtsova}, we know that, 
\begin{equation}
\label{eq: eq 2}
    \Delta_n(s) \leq A \sum_{k=1}^n \left[\frac{\E W_k^2 \I(\abs{W_k} > (1+\abs{s})B_{n,W})}{(1+\abs{s})^2B_{n,W}^2} + \frac{\E \abs{W_k}^3 \I(\abs{W_k} \leq (1+\abs{s})B_{n,W})}{(1+\abs{s})^3B_{n,W}^3} \right]
\end{equation}
for an absolute constant $A$.
Then, we have
\begin{equation}\label{eq:eq_3}
    \begin{split}
     &\abs{\E\sigma\left(\frac{S_{n,W}}{B_{n,W}} - \frac{t}{B_{n,W}}\right) - \E\sigma\left(Z- \frac{t}{B_{n,W}}\right)}\\ 
     & \leq \int_{0}^{\infty} \Delta_n(s) ds\\
     & = A \sum_{k=1}^n  \int_{\frac{t}{B_{n,W}}}^{\infty} \Bigg[\frac{\E W_k^2 \I(\abs{W_k} > (1+\abs{s})B_{n,W}}{(1+\abs{s})^2B_{n,W}^2} +  \frac{\E \abs{W_k}^3 \I(\abs{W_k} \leq (1+\abs{s})B_{n,W})}{(1+\abs{s})^3B_{n,W}^3} \Bigg] ds \\
     & = A \sum_{k=1}^n  \int_{\frac{t}{B_{n,W}}}^{\infty} \Bigg[\frac{\E W_k^2 \I(\abs{W_k} > (1+s)B_{n,W}}{(1+s)^2B_{n,W}^2} +  \frac{\E \abs{W_k}^3 \I(\abs{W_k} \leq (1+s)B_{n,W})}{(1+s)^3B_{n,W}^3} \Bigg] ds.
    \end{split}  
\end{equation}
Now, the first term on the right hand side can be simplified as follows:
\begin{equation}
\label{eq:eq_4}
     \begin{split}
     &\int_{\frac{t}{B_{n,W}}}^{\infty}\frac{\E W_k^2 \I(\abs{W_k} > (1+s)B_{n,W})}{(1+s)^2B_{n,W}^2} ds\\ 
     & = \E\int_{\frac{t}{B_{n,W}}}^{\infty}\frac{W_k^2 \I(\abs{W_k} > (1+s)B_{n,W})}{(1+s)^2B_{n,W}^2} ds\\
     & = \frac{1}{B_{n,W}^2}\E W_k^2 \I\left(\abs{W_k}B_{n,W}^{-1}-1 \geq \frac{t}{B_{n,W}}\right) \int_{\frac{t}{B_{n,W}}}^{\abs{W_k}B_{n,W}^{-1}-1}\frac{1}{(1+s)^2} ds \\
     & = \frac{1}{B_{n,W}^2}\E W_k^2\I(\abs{W_k} \geq t + B_{n,W})  \int_{\frac{t}{B_{n,W}}}^{\abs{W_k}B_{n,W}^{-1}-1}\frac{1}{(1+s)^2} ds \\
     & = \frac{1}{B_{n,W}^2}\E W_k^2 \I(\abs{W_k} \geq t + B_{n,W}) \Bigg[-\frac{1}{1+s}\Bigg]_0^{\abs{W_k}B_{n,W}^{-1}-1}\\
     & = \frac{1}{B_{n,W}^2}\E W_k^2 \I(\abs{W_k} \geq t + B_{n,W}) \Bigg[1 - \frac{1}{1 + (\abs{W_k}B_{n,W}^{-1}-1)}\Bigg]\\
     & = \frac{1}{B_{n,W}^2}\E W_k^2 \I(\abs{W_k} \geq t + B_{n,W}) \Bigg[1 - \frac{B_{n,W}}{\abs{W_k}}\Bigg] \\
     & = \frac{1}{B_{n,W}^2}\E W_k^2\I(\abs{W_k} \geq t + B_{n,W})  - \frac{1}{B_{n,W}}\E \abs{W_k}\I(\abs{W_k} \geq t + B_{n,W}). 
    \end{split}  
\end{equation}
The second term on the right-hand side of the bound on $|\mathbb{E}\sigma(S_{n,W}/B_{n,W} - t/B_{n,W}) - \mathbb{E}\sigma(Z - t/B_{n,W})|$ can be simplified as follows:
\begin{equation}\label{eq:eq_5}
     \begin{split}
     &\int_{\frac{t}{B_{n,W}}}^{\infty}\frac{\E \abs{W_k}^3 \I(\abs{W_k} \leq (1+s)B_{n,W})}{(1+s)^3B_{n,W}^3} ds\\ 
     & = \E\int_{\frac{t}{B_{n,W}}}^{\infty}\frac{\abs{W_k}^3 \I(\abs{W_k} \leq (1+s)B_{n,W})}{(1+s)^3B_{n,W}^3} ds\\
     & = \frac{1}{B_{n,W}^3}\E \abs{W_k}^3 \int_{\max\{tB_{n,W}^{-1},\abs{W_k}B_{n,W}^{-1}-1\}}^\infty \frac{1}{(1+s)^3} ds \\
     & =  \frac{1}{B_{n,W}^3}\E \abs{W_k}^3 \Bigg[-\frac{1}{2(1+s)^2}\Bigg]_{\max\{tB_{n,W}^{-1},\abs{W_k}B_{n,W}^{-1}-1\}}^\infty \\
     & = \frac{1}{2B_{n,W}^3}\E \abs{W_k}^3 \frac{1}{(1+\abs{W_k}B_{n,W}^{-1}-1)^2}\I(\abs{W_k}B_{n,W}^{-1}-1\geq tB_{n,W}^{-1})\\
     &\qquad+ \frac{1}{2B_{n,W}^3}\E \abs{W_k}^3 \frac{1}{(1+tB_{n,W}^{-1})^2}\I(\abs{W_k}B_{n,W}^{-1}-1< tB_{n,W}^{-1}) \\
     & = \frac{1}{2B_{n,W}}\E \abs{W_k}\I(\abs{W_k} \geq t + B_{n,W}) + \frac{1}{2B_{n,W}}\E \abs{W_k}^3 \frac{1}{(t+B_{n,W})^2}\I(\abs{W_k} < t+B_{n,W})
    \end{split}  
\end{equation}
Combining these inequalities, we have 
\begin{equation}
\label{eq: eq 6}
    \begin{split}
     &\abs{\E\sigma\left(\frac{S_{n,W}}{B_{n,W}} - \frac{t}{B_{n,W}}\right) - \E\sigma\left(Z- \frac{t}{B_{n,W}}\right)}\\ 
     & \leq A \sum_{k=1}^n  \int_{tB_{n,W}^{-1}}^{\infty} \Bigg[\frac{\E W_k^2 \I(\abs{W_k} > (1+s)B_{n,W}}{(1+s)^2B_{n,W}^2} +  \frac{\E \abs{W_k}^3 \I(\abs{W_k} \leq (1+s)B_{n,W})}{(1+s)^3B_{n,W}^3} \Bigg] ds \\
     & \leq \frac{A}{B_{n,W}^2}\sum_{k=1}^n \E W_k^2\I(\abs{W_k} \geq t + B_{n,W})  -   \frac{A}{2B_{n,W}}\sum_{k=1}^n\E \abs{W_k}\I(\abs{W_k} \geq t + B_{n,W}) + \\
     &\quad\qquad \frac{A}{2B_{n,W}}\sum_{k=1}^n \frac{\E \abs{W_k}^3\I(\abs{W_k} < t+B_{n,W})}{(t+B_{n,W})^2} \\
     &\leq \frac{A}{B_{n,W}^2}\sum_{k=1}^n \E W_k^2\I(\abs{W_k} \geq t + B_{n,W}) + 
     \frac{A}{2B_{n,W}}\sum_{k=1}^n \frac{\E \abs{W_k}^3
\I(\abs{W_k} < t+B_{n,W})}{(t+B_{n,W})^2} 
    \end{split}  
\end{equation}

Finally, using the fact that $\sigma(\cdot)$ is a positively homogeneous function (i.e., $\sigma(cx) = c\sigma(x)$ for all $c > 0$), we obtain
\begin{equation}
    \label{eq: eq 7}
    \begin{split}
     &\abs{\E\sigma(S_{n,W}-t) - \E\sigma(B_{n,W}Z-t)}\\ 
     & = \abs{\E\max\{0,S_{n,W}-t\} - \E\max\{0,B_{n,W}Z-t\}} \\
     & = \abs{\E B_{n,W}\max\Bigg\{0,\frac{S_{n,W}}{B_{n,W}} - \frac{t}{B_{n,W}}\Bigg\} - \E B_{n,W}\max\Bigg\{0,\left(Z- \frac{t}{B_{n,W}}\right)\Bigg\}} \\
     & = B_{n,W}\abs{\E\sigma\left(\frac{S_{n,W}}{B_{n,W}} - \frac{t}{B_{n,W}}\right) - \E\sigma\left(Z- \frac{t}{B_{n,W}}\right)} \\
      &  \leq \frac{A}{B_{n,W}}\sum_{k=1}^n \E W_k^2\I(\abs{W_k} \geq t + B_{n,W})  +   
     \frac{A}{2}\sum_{k=1}^n \frac{\E \abs{W_k}^3\I(\abs{W_k} < t+B_{n,W})}{(t+B_{n,W})^2}  \quad [\text{using }\cref{eq: eq 6}]
    \end{split}
\end{equation}
This concludes the proof. 
\end{proof}

\begin{theorem}
\label{bound for squared relu function}
Let $\sigma(\cdot)$ denote the ReLU function, i.e., $\sigma(x) = \max\{0,x\}$. Let $W_1, W_2, \dots , W_n$ be independent univariate random variables 
such that $\E W_k = 0$ and $\sigma_k^2 = \E W_k^2 < \infty$, $\forall k = 1, 2, \dots, n$. Let 
\begin{equation*}
    S_{n,W}\coloneqq \sum_{k=1}^n W_k, \quad B_{n,W}^2 \coloneqq \sum_{k=1}^n \sigma_k^2.
\end{equation*}
Let $Z$ be a standard normal random variable. Then, for $t\geq0$, 
\begin{equation*}
\begin{split}
      \abs{\E\sigma^2(S_{n,W}-t) - \E\sigma^2(B_{n,W}Z-t)}  &\leq  2A\sum_{k=1}^n \E W_k^2\I(\abs{W_k} \geq t + B_{n,W}) \ln(\abs{W_k})  \\
      &\qquad+ 2A\left[1+\ln(1 + tB_{n,W}^{-1})\right] \sum_{k=1}^n \E W_k^2\I(\abs{W_k} \geq t + B_{n,W})\\  
      &\qquad+ 2 \frac{A}{B_{n,W}(1+tB_{n,W}^{-1})} \sum_{k=1}^n \E \abs{W_k}^3 \I(\abs{W_k} < t + B_{n,W})
      \end{split}
\end{equation*}
for an absolute constant $A$.
\end{theorem}

\begin{proof}
For $t \geq 0$,
\begin{equation*}
    \begin{split}
    &\abs{\E\sigma^2\left(\frac{S_{n,W}}{B_{n,W}} - \frac{t}{B_{n,W}}\right) - \E\sigma^2\left(Z- \frac{t}{B_{n,W}}\right)}\\ & = \abs{\int_{0}^{\infty}\left[\P\left(\sigma^2\left(\frac{S_{n,W}}{B_{n,W}} - \frac{t}{B_{n,W}}\right) \geq w \right) - \P\left(\sigma^2\left(Z- \frac{t}{B_{n,W}}\right) \geq w \right) \right]dw} \\
     & = \abs{\int_{0}^{\infty}\left[\P\left(\sigma^2\left(Z- \frac{t}{B_{n,W}}\right) < w \right) - \P\left(\sigma^2\left(\frac{S_{n,W}}{B_{n,W}} - \frac{t}{B_{n,W}}\right) < w \right) \right]dw} \\
     & \leq  \int_{0}^{\infty}\abs{\left[\P\left(\sigma^2\left(Z- \frac{t}{B_{n,W}}\right) < w \right) - \P\left(\sigma^2\left(\frac{S_{n,W}}{B_{n,W}} - \frac{t}{B_{n,W}}\right) < w \right) \right]dw}  \\
     & = \int_{0}^{\infty}\Bigg |\Bigg[\P\left(\max\Bigg\{0,Z-\frac{t}{B_{n,W}}\Bigg\}^2 < w \right) - \P\left(\max\Bigg\{0,\frac{S_{n,W}}{B_{n,W}} - \frac{t}{B_{n,W}}\Bigg\}^2 < w \right) \Bigg]dw \Bigg| 
    \end{split}
\end{equation*}

Since we are considering the integral from  $0$ to $\infty$,  so we have $w > 0$. So, for a random variable $X$, $\sqrt{w} > \max\{0,X\}$ means that $w > 0$ and $\sqrt{w} > X$. But because of the lower limit of the integral, we already have $w > 0$. So, the event $w > 0$ and $\sqrt{w} > X$ boils down to the event $\sqrt{w} > X$.

\begin{equation}
\label{eq: eq 8}
    \begin{split}
        &\abs{\E\sigma^2\left(\frac{S_{n,W}}{B_{n,W}} - \frac{t}{B_{n,W}}\right) - \E\sigma^2\left(Z- \frac{t}{B_{n,W}}\right)}\\ 
        & \leq 
         \int_{0}^{\infty}\Bigg|\Bigg[\P\left(Z < \frac{t}{B_{n,W}} + \sqrt{w} \right) - \P\left(\frac{S_n}{B_n} < \frac{t}{B_{n,W}} +  \sqrt{w} \right)\Bigg]dw\Bigg|  \\
         & = \int_{0}^{\infty}  \Delta_n\left(\frac{t}{B_{n,W}} + \sqrt{w}\right) dw  \\
         & = 2 \int_{t/B_{n,W}}^{\infty} \Delta_n(s)\Bigg(s - \frac{t}{B_{n,W}}\Bigg)ds \\
         & [\text{by using the change of variables } s = \frac{t}{B_{n,W}} + \sqrt{w}]  \\ 
         & = 2 \int_{{t}/{B_{n,w}}}^{\infty} s\Delta_n(s)ds - 2\frac{t}{B_{n,W}}\int_{{t}/{B_{n,w}}}^{\infty} \Delta_n(s)ds \\
         & \leq 2 \int_{{t}/{B_{n,w}}}^{\infty} s\Delta_n(s)ds \quad [\because \Delta_n(s) > 0,  \forall s, t \geq 0, B_{n,W}> 0]
    \end{split}
\end{equation}

By \Cref{4 of Shevtsova}, we know that, 
\begin{equation}
\label{eq: eq 9}
    \Delta_n(s) \leq A \sum_{k=1}^n \left[\frac{\E W_k^2 \I(\abs{W_k} > (1+\abs{s})B_{n,W})}{(1+\abs{s})^2B_{n,W}^2} + \frac{\E \abs{W_k}^3 \I(\abs{W_k} \leq (1+\abs{s})B_{n,W})}{(1+\abs{s})^3B_{n,W}^3} \right]
\end{equation}
for an absolute constant $A$.
We can see that the bound of $\Delta_n(s)$ given in \cref{eq: eq 9} is an even function of $s$. So, continuing from \cref{eq: eq 8}, we can write 
{\small\begin{equation}
\label{eq: eq 10}
    \begin{split}
     2 \int_{tB_{n,w}^{-1}}^{\infty} s\Delta_n(s)ds  \leq 2 \int_{tB_{n,w}^{-1}}^{\infty} s A \sum_{k=1}^n \Bigg[\frac{\E W_k^2 \I(\abs{W_k} > (1+\abs{s})B_{n,W})}{(1+\abs{s})^2B_{n,W}^2} + \frac{\E \abs{W_k}^3 \I(\abs{W_k} \leq (1+\abs{s})B_{n,W})}{(1+\abs{s})^3B_{n,W}^3} \Bigg] ds \\
     = 2 A \sum_{k=1}^n \E\int_{tB_{n,w}^{-1}}^{\infty} s  \Bigg[\frac{W_k^2 \I(\abs{W_k} > (1+s)B_{n,W})}{(1+s)^2B_{n,W}^2} + \frac{\abs{W_k}^3 \I(\abs{W_k} \leq (1+s)B_{n,W})}{(1+s)^3B_{n,W}^3} \Bigg] ds \\ 
     = 2 \frac{A}{B_{n,W}^2} \sum_{k=1}^n \E W_k^2\I(\abs{W_k} \geq t + B_{n,W})\int_{tB_{n,w}^{-1}}^{\abs{W_k}B_{n,w}^{-1}-1} \frac{s}{(1+s)^2}ds + 2 \frac{A}{B_{n,W}^3} \sum_{k=1}^n \E \abs{W_k}^3\int_{\max\{tB_{n,w}^{-1},\abs{W_k}B_{n,w}^{-1}-1\}}^{\infty} \frac{s}{(1+s)^3}ds \\
     \leq 2 \frac{A}{B_{n,W}^2} \sum_{k=1}^n \E W_k^2\I(\abs{W_k} \geq t + B_{n,W}) \int_{tB_{n,w}^{-1}}^{\abs{W_k}B_{n,w}^{-1}-1} \frac{1}{(1+s)}ds + 2 \frac{A}{B_{n,W}^3} \sum_{k=1}^n \E \abs{W_k}^3\int_{\max\{tB_{n,w}^{-1},\abs{W_k}B_{n,w}^{-1}-1\}}^{\infty} \frac{1}{(1+s)^2}ds\\
     = 2 \frac{A}{B_{n,W}^2} \sum_{k=1}^n \E W_k^2\I(\abs{W_k} \geq t + B_{n,W}) [\ln(1+s)]_{tB_{n,w}^{-1}}^{\abs{W_k}B_{n,w}^{-1}-1} + 2 \frac{A}{B_{n,W}^3} \sum_{k=1}^n \E \abs{W_k}^3\left[-\frac{1}{1+s}\right]_{\max\{tB_{n,w}^{-1},\abs{W_k}B_{n,w}^{-1}-1\}}^{\infty}  \\
     = 2 \frac{A}{B_{n,W}^2} \sum_{k=1}^n \E W_k^2\I(\abs{W_k} \geq t + B_{n,W}) \ln(\abs{W_k}) - 2 \frac{A\ln(B_{n,W})}{B_{n,W}^2} \sum_{k=1}^n \E W_k^2\I(\abs{W_k} \geq t + B_{n,W}) + \\
     2 \frac{A\ln(1 + tB_{n,W}^{-1})}{B_{n,W}^2} \sum_{k=1}^n \E W_k^2\I(\abs{W_k} \geq t + B_{n,W})  + 2 \frac{A}{B_{n,W}^3(1+tB_{n,W}^{-1})} \sum_{k=1}^n \E \abs{W_k}^3 \I(\abs{W_k} < t + B_{n,W}) + \\
     2 \frac{A}{B_{n,W}^2} \sum_{k=1}^n \E W_k^2 \I(\abs{W_k} \geq t + B_{n,W}) \\
     \leq 2 \frac{A}{B_{n,W}^2} \sum_{k=1}^n \E W_k^2\I(\abs{W_k} \geq t + B_{n,W}) \ln(\abs{W_k}) + 2 \frac{A\ln(1 + tB_{n,W}^{-1})}{B_{n,W}^2} \sum_{k=1}^n \E W_k^2\I(\abs{W_k} \geq t + B_{n,W})  + \\
     2 \frac{A}{B_{n,W}^3(1+tB_{n,W}^{-1})} \sum_{k=1}^n \E \abs{W_k}^3 \I(\abs{W_k} < t + B_{n,W}) + 2 \frac{A}{B_{n,W}^2} \sum_{k=1}^n \E W_k^2 \I(\abs{W_k} \geq t + B_{n,W})
    \end{split}
\end{equation}}

Finally, using the fact that $\sigma^2(cx) = c^2\sigma(x)$ for all $c > 0$), we obtain 
\begin{equation}
    \label{eq: eq 11}
    \begin{split}
     \abs{\E\sigma^2(S_{n,W}-t) - \E\sigma^2(B_{n,W}Z-t)}   = \abs{\E\max\{0,S_{n,W}-t\}^2 - \E\max\{0,B_{n,W}Z-t\}^2} \\
     = \abs{\E B_{n,W}^2\max\Bigg\{0,\frac{S_{n,W}}{B_{n,W}} - \frac{t}{B_{n,W}}\Bigg\}^2 - \E B_{n,W}^2\max\Bigg\{0,\left(Z- \frac{t}{B_{n,W}}\right)\Bigg\}^2} \\
     = B_{n,W}^2\abs{\E\sigma^2\left(\frac{S_{n,W}}{B_{n,W}} - \frac{t}{B_{n,W}}\right) - \E\sigma^2\left(Z- \frac{t}{B_{n,W}}\right)} \\
     \leq  2A\sum_{k=1}^n \E W_k^2\I(\abs{W_k} \geq t + B_{n,W}) \ln(\abs{W_k}) + 2A\ln(1 + tB_{n,W}^{-1}) \sum_{k=1}^n \E W_k^2\I(\abs{W_k} \geq t + B_{n,W})  + \\
     2 \frac{A}{B_{n,W}(1+tB_{n,W}^{-1})} \sum_{k=1}^n \E \abs{W_k}^3 \I(\abs{W_k} < t + B_{n,W}) + 2A\sum_{k=1}^n \E W_k^2 \I(\abs{W_k} \geq t + B_{n,W}) \\
     = 2A\sum_{k=1}^n \E W_k^2\I(\abs{W_k} \geq t + B_{n,W}) \ln(\abs{W_k}) + 2A\left[1+\ln(1 + tB_{n,W}^{-1})\right] \sum_{k=1}^n \E W_k^2\I(\abs{W_k} \geq t + B_{n,W})  + \\
     2 \frac{A}{B_{n,W}(1+tB_{n,W}^{-1})} \sum_{k=1}^n \E \abs{W_k}^3 \I(\abs{W_k} < t + B_{n,W})
    \end{split}
\end{equation}
This concludes the proof.
\end{proof}

\subsection{Bounds for single-layer neural networks}
\paragraph{Motivation.} In the sections that follow, we derive bounds for functions that have an integral representation of the forms given in \cite{klusowski2018approximation}. We provide the statements and proofs of the existence results and bounds for functions with bounded $l_1$ norm of inner parameters. As mentioned in \cref{sec: Problem Statement}, our main goal is to find bounds for $\abs{Ef(S_{n,X}) - Ef(Z)}$ for arbitrary functions $f$ and neural networks are one of the most prominent classes of functions that are dense. The single layer neural networks that we consider are dense in $L_2$ and can be used to approximate any continuous function. 
\begin{theorem}
\label{bound for functions which have an integral representation of the form given in Theorem 1 of klusowski2018approximation}
Suppose $X_1, \dots, X_n$ is a sequence of mean zero independent $d$-dimensional random vectors. Set $S_{n,X} = n^{-1/2}\sum_{i=1}^n X_i$. Let $Z$ be a $d$-dimensional Gaussian random vector with mean zero and variance-covariance matrix given by $\Sigma = \mbox{Var}(S_{n,X})$. 
Let $f$ be a function that admits an integral representation of the form 
\begin{equation}
\label{eq: eq 9.1}
    f(x) = v\int_{[0,1]\times\{a:\norm{a}_1=1\}}\eta(t,a)(a^\top x - t)^{s-1}_+dP(t,a),
\end{equation}
for $x \in D = [-1,1]^d$ and $s \in \{2,3\}$, where $P$ is a probability measure on $[0,1]\times\{a:\norm{a}_1=1\}$ and $\eta(t,a)$ is either $-1$ or $+1$ for all $t, a$.
Then there exists an universal constant $A > 0$ such that
\begin{itemize}
    \item if $s = 2$, then
    \begin{equation*}
    \begin{split}
    \abs{\E f(S_{n,X}) - \E f(Z)} & \leq \frac{A\abs{v}}{n}\Bigg\{\sup_{a:\norm{a}_1=1}\left[\frac{1}{\norm{a}_\Sigma}\sum_{k=1}^n \E  (a^\top X_k)^2 \I(\abs{a^\top X_k} \geq \sqrt{n}\norm{a}_\Sigma)\right] + \\ 
      & \sup_{a:\norm{a}_1=1}\left[\frac{1}{2\norm{a}_\Sigma^2}\sum_{k=1}^n\E \frac{\abs{a^\top X_k}^3}{\sqrt{n}}\I(\abs{a^\top X_k} <  \sqrt{n}\norm{a}_\Sigma)\right]\Bigg\},
      \end{split}
    \end{equation*}
    and
    \item if $s = 3$, then
    \begin{equation*}
        \begin{split}
          \abs{\E f(S_{n,X}) - \E f(Z)} & \leq 2\frac{A\abs{v}}{n}\Bigg\{\sup_{a:\norm{a}_1=1}\sum_{k=1}^n \E \left[(a^\top X_k)^2 \ln\left(\frac{e\abs{a^\top X_k}}{\norm{a}_\Sigma}\right)\right]\I(\abs{a^\top X_k} \geq \sqrt{n}\norm{a}_\Sigma) + \\
      & \sup_{a:\norm{a}_1=1}\frac{1}{\norm{a}_\Sigma} \sum_{k=1}^n\E\frac{\abs{a^\top X_k}^3}{\sqrt{n}}\I(\abs{a^\top _k} < \sqrt{n}\norm{a}_\Sigma)  \Bigg\}.
    \end{split}
\end{equation*}
\end{itemize}
\end{theorem}
\begin{proof}
Let $W_i = {a^{\top}X_i}/{\sqrt{n}}$ for $i = 1, \dots, n$. We define the following notations:
\begin{equation*}
        S_{n,X} \coloneqq \frac{1}{\sqrt{n}}\sum_{i=1}^n X_i, \quad \Sigma =  \mathrm{Var}\left(\frac{1}{\sqrt{n}}\sum_{i=1}^n X_i\right), \quad \Tilde{S}_{n,W}\coloneqq \sum_{i=1}^n W_i, \quad B_{n,W}^2 = \sum_{i=1}^n \mathrm{Var}(W_i).
\end{equation*}

We have, 
\begin{equation*}
    a^{\top}S_{n,X} = a^{\top} \left(\frac{1}{\sqrt{n}}\sum_{i=1}^n X_i \right) = \sum_{i=1}^n \frac{a^{\top}X_i}{\sqrt{n}} = \sum_{i=1}^n W_i = \Tilde{S}_{n,W}
\end{equation*}

Also, 
\begin{equation*}
    \begin{split}
        & Z \sim \gauss(0,\mathrm{Var}(S_{n,X})) \\
        \implies & Z \sim \gauss\left(0,\mathrm{Var}\left(\frac{1}{\sqrt{n}}\sum_{i=1}^n X_i\right) \right)  \\
        \implies  & Z \sim \gauss\left(0, \frac{1}{n} \sum_{i=1}^n \mathrm{Var}(X_i) \right)  \quad [\because X_i \text{'s are independent}] \\
        \implies  & a^{\top}Z \sim \gauss\left(0, \frac{a^{\top}a}{n} \sum_{i=1}^n \mathrm{Var}(X_i) \right) \\
        \implies  & a^{\top}Z \sim \gauss\left(0, \sum_{i=1}^n \mathrm{Var}\left(\frac{a^{\top}X_i}{\sqrt{n}}\right) \right) \\
        \implies  & a^{\top}Z \sim \gauss\left(0,\sum_{i=1}^n \mathrm{Var}(W_i)\right) \\
        \implies  & a^{\top}Z \sim \gauss(0,B_{n,W}^2) 
    \end{split}
\end{equation*}

Now, 
\begin{equation}
\label{eq: eq 9.1}
    \begin{split}
      \abs{\E f(S_{n,X}) - \E f(Z)} & = \abs{v\int_{[0,1]\times\{a:\norm{a}=1\}}\eta(t,a)\big[\E(a^{\top}S_{n,X} - t)^{s-1}_+ - \E(a^{\top}Z - t)^{s-1}_+\big]dP(t,a)} \\
      & = \abs{v\int_{[0,1]\times\{a:\norm{a}=1\}}\eta(t,a)\big[\E\sigma(a^{\top}S_{n,X} - t)^{s-1} - \E\sigma(a^{\top}Z - t)^{s-1}\big]dP(t,a)} \\
      & \leq \abs{v}\int_{[0,1]\times\{a:\norm{a}=1\}}\abs{\eta(t,a)}\abs{\E\sigma(a^{\top}S_{n,X} - t)^{s-1} - \E\sigma(a^{\top}Z - t)^{s-1}}dP(t,a) \\
      & [\eta(t,a) \in {-1,1} \implies \abs{\eta(t,a)} = 1 \forall t \in [0,1], a \in \{a:\norm{a}=1\}] \\
      & = \abs{v} \int_{[0,1]\times\{a:\norm{a}=1\}}\abs{\E\sigma(a^{\top}S_{n,X} - t)^{s-1} - \E\sigma(a^{\top}Z - t)^{s-1}}dP(t,a)
    \end{split}
\end{equation}

For s = 2, continuing from \cref{eq: eq 9.1}, we have, 

\begin{equation}
\label{eq: eq 9.2}
    \begin{split}
      \abs{\E f(S_{n,X}) - \E f(Z)} & \leq \abs{v} \int_{[0,1]\times\{a:\norm{a}_1=1\}}\abs{\E\sigma(a^{\top}S_{n,X} - t) - \E\sigma(a^{\top}Z - t)}dP(t,a) \\
      & \leq \abs{v} \int_{[0,1]\times\{a:\norm{a}_1=1\}}\abs{\E\sigma(\Tilde{S}_{n,W}- t) - \E\sigma(a^{\top}Z - t)}dP(t,a) \\
      & \leq \abs{v}\int_{[0,1]\times\{a:\norm{a}_1=1\}}\left[\frac{A}{B_{n,w}}\sum_{k=1}^n \E W_k^2\I(\abs{W_k} \geq B_{n,W}) + \frac{A}{2B_{n,W}^2}\sum_{k=1}^n\E \abs{W_k}^3\I(\abs{W_k} < B_{n,W})\right] dP(t,a)  \\
      & [\text{using }\cref{bound for relu function}] \\
      & \leq \abs{v}\sup_{[0,1]\times\{a:\norm{a}_1=1\}}\left[\frac{A}{B_{n,w}}\sum_{k=1}^n \E W_k^2\I(\abs{W_k} \geq B_{n,W}) + \frac{A}{2B_{n,W}^2}\sum_{k=1}^n\E \abs{W_k}^3\I(\abs{W_k} < B_{n,W})\right]\\
      & =  A\abs{v}\sup_{[0,1]\times\{a:\norm{a}_1=1\}}\left[\frac{1}{B_{n,w}}\sum_{k=1}^n \E W_k^2\I(\abs{W_k} \geq B_{n,W}) + \frac{1}{2B_{n,W}^2}\sum_{k=1}^n\E \abs{W_k}^3\I(\abs{W_k} < B_{n,W})\right]\\
      & =  A\abs{v}\sup_{a:\norm{a}_1=1}\left[\frac{1}{B_{n,w}}\sum_{k=1}^n \E W_k^2\I(\abs{W_k} \geq B_{n,W}) + \frac{1}{2B_{n,W}^2}\sum_{k=1}^n\E \abs{W_k}^3\I(\abs{W_k} < B_{n,W})\right] \\
      & = A\abs{v}\sup_{a:\norm{a}_1=1}\Bigg[\frac{1}{\sqrt{a^\top\Sigma a}}\sum_{k=1}^n \E \frac{(a^\top X_k)^2}{n}\I(\abs{a^\top X_k} \geq \sqrt{na^\top\Sigma a} ) + \\ 
      & \frac{1}{2a^\top\Sigma a}\sum_{k=1}^n\E \frac{\abs{a^\top X_k}^3}{n\sqrt{n}}\I(\abs{a^\top X_k} <  \sqrt{na^\top\Sigma a})\Bigg]\\
      & = \frac{A\abs{v}}{n}\sup_{a:\norm{a}_1=1}\Bigg[\frac{1}{\norm{a}_\Sigma}\sum_{k=1}^n \E  (a^\top X_k)^2 \I(\abs{a^\top X_k} \geq \sqrt{n}\norm{a}_\Sigma) + \\ 
      & \frac{1}{2\norm{a}_\Sigma^2}\sum_{k=1}^n\E \frac{\abs{a^\top X_k}^3}{\sqrt{n}}\I(\abs{a^\top X_k} <  \sqrt{n}\norm{a}_\Sigma)\Bigg]\\
      & \leq \frac{A\abs{v}}{n}\Bigg\{\sup_{a:\norm{a}_1=1}\left[\frac{1}{\norm{a}_\Sigma}\sum_{k=1}^n \E  (a^\top X_k)^2 \I(\abs{a^\top X_k} \geq \sqrt{n}\norm{a}_\Sigma)\right] + \\ 
      & \sup_{a:\norm{a}_1=1}\left[\frac{1}{2\norm{a}_\Sigma^2}\sum_{k=1}^n\E \frac{\abs{a^\top X_k}^3}{\sqrt{n}}\I(\abs{a^\top X_k} <  \sqrt{n}\norm{a}_\Sigma)\right]\Bigg\}\\
      & \Bigg[\text{where, } A \text{ is an absolute constant, } \Sigma =  \mathrm{Var}\left(\frac{1}{\sqrt{n}}\sum_{i=1}^n X_i\right) \text{ and } \norm{a}_\Sigma = \sqrt{a^\top\Sigma a}\Bigg] 
    \end{split}
\end{equation}

For s = 3, continuing from \cref{eq: eq 9.1}, we have, 

\begin{equation}
\label{eq: eq 9.3}
    \begin{split}
      \abs{\E f(S_{n,X}) - \E f(Z)}  & \leq \abs{v} \int_{[0,1]\times\{a:\norm{a}_1=1\}}\abs{\E\sigma^2(a^{\top}S_{n,X} - t) - \E\sigma^2(a^{\top}Z - t)}dP(t,a) \\
      & \abs{v} \int_{[0,1]\times\{a:\norm{a}_1=1\}}\abs{\E\sigma^2(\Tilde{S}_{n,W}- t) - \E\sigma^2(a^{\top}Z - t)}dP(t,a) \\
      & \leq \abs{v}\int_{[0,1]\times\{a:\norm{a}_1=1\}}\Bigg\{2A \sum_{k=1}^n \E W_k^2\left[\ln\frac{\abs{W_k}}{B_{n,W}}+1\right]\I(\abs{W_k} \geq B_{n,W}) + \\
      & 2 \frac{A}{B_{n,W}} \sum_{k=1}^n\E\abs{W_k}^3\I(\abs{W_k} < B_{n,W})  \Bigg\} dP(t,a)  \quad [\text{using }\cref{bound for squared relu function} ]  \\
      & = 2A\abs{v}\sup_{a:\norm{a}_1=1}\Bigg\{\sum_{k=1}^n \E W_k^2\left[\ln\frac{\abs{W_k}}{B_{n,W}}+1\right]\I(\abs{W_k} \geq B_{n,W}) + \\
      & \frac{1}{B_{n,W}} \sum_{k=1}^n\E\abs{W_k}^3\I(\abs{W_k} < B_{n,W})  \Bigg\}\\
      & = 2A\abs{v}\sup_{a:\norm{a}_1=1}\Bigg\{\sum_{k=1}^n \E \frac{(a^\top X_k)^2}{n} \left[\ln\frac{\abs{a^\top X_k}}{\norm{a}_\Sigma}+1\right]\I(\abs{a^\top X_k} \geq \sqrt{n}\norm{a}_\Sigma) + \\
      & \frac{1}{\norm{a}_\Sigma} \sum_{k=1}^n\E\frac{\abs{a^\top X_k}^3}{n\sqrt{n}}\I(\abs{a^\top X_k} < \sqrt{n}\norm{a}_\Sigma)  \Bigg\}\\
      & = 2\frac{A\abs{v}}{n}\sup_{a:\norm{a}_1=1}\Bigg\{\sum_{k=1}^n \E (a^\top X_k)^2 \left[\ln\frac{\abs{a^\top X_k}}{\norm{a}_\Sigma}+1\right]\I(\abs{a^\top X_k} \geq \sqrt{n}\norm{a}_\Sigma) + \\
      & \frac{1}{\norm{a}_\Sigma} \sum_{k=1}^n\E\frac{\abs{a^\top X_k}^3}{\sqrt{n}}\I(\abs{a^\top X_k} < \sqrt{n}\norm{a}_\Sigma)  \Bigg\}\\
      & \leq  2\frac{A\abs{v}}{n}\Bigg\{\sup_{a:\norm{a}_1=1}\sum_{k=1}^n \E \left[(a^\top X_k)^2 \ln\left(\frac{e\abs{a^\top X_k}}{\norm{a}_\Sigma}\right)\right]\I(\abs{a^\top X_k} \geq \sqrt{n}\norm{a}_\Sigma) + \\
      & \sup_{a:\norm{a}_1=1}\frac{1}{\norm{a}_\Sigma} \sum_{k=1}^n\E\frac{\abs{a^\top X_k}^3}{\sqrt{n}}\I(\abs{a^\top _k} < \sqrt{n}\norm{a}_\Sigma)  \Bigg\}\\
      & \Bigg[\text{where, } A \text{ is an absolute constant, } \Sigma =  \mathrm{Var}\left(\frac{1}{\sqrt{n}}\sum_{i=1}^n X_i\right) \text{ and } \norm{a}_\Sigma = \sqrt{a^\top\Sigma a}\Bigg] 
    \end{split}
\end{equation}
\end{proof}

\begin{theorem}
\label{Theorem 2 of klusowski2018approximation}
Suppose $X_1, \dots, X_n$ is a sequence of mean zero independent $d$-dimensional random vectors. Set $S_{n,X} = n^{-1/2}\sum_{i=1}^n X_i$. Let $Z$ be a $d$-dimensional Gaussian random vector with mean zero and variance-covariance matrix given by $\Sigma = \mathrm{Var}(S_{n,X})$. Let $D = [-1,1]^d$. Suppose $f: D \longrightarrow \R$ admits a Fourier representation $f(x) = \int_{\R^d}e^{ix^\top \omega} \mathcal{F}(f)(\omega)d\omega$ and
\begin{equation*}
    v_{f,2} = \int_{\R^d} \norm{\omega}_1^2\abs{\mathcal{F}(f)(\omega)}d\omega < \infty.
\end{equation*}
Then, 
$\exists$ a probability measure $P$ on $[0,1]\times \{a: \norm{a} = 1\}$, $\eta \in \{\pm 1\}$, $s = 2$ and $v$ such that $\abs{v} \leq 2v_{f,2}$, such that,
\begin{equation*}
    f(x) - f(0) - x\cdot \nabla f(0) = v \int_{[0,1] \times \{a: \norm{a} = 1\}} \eta(t,a) (a\cdot x - t)_+^{s-1} dP(t,a) \quad \forall x \in D. 
\end{equation*}
and hence, we have,
\begin{equation*}
\begin{split}
    \abs{\E f(S_{n,X}) - \E f(Z)} & \leq  2\frac{Av_{f,2}}{n}\Bigg\{ \sup_{a:\norm{a}=1} \left[\frac{1}{\norm{a}_\Sigma} \sum_{k=1}^{n}\E(a^\top X_k)^2 \I\left(\abs{a^\top X_k} \ge \sqrt{n}\norm{a}_\Sigma\right)\right] \\
    & + \sup_{a:\norm{a}=1} \left[\frac{1}{2\norm{a}_\Sigma^2} \sum_{k=1}^{n}\E\frac{\abs{a^\top X_k}^3}{\sqrt{n}}\I\left(\abs{a^\top X_k} < \sqrt{n}\norm{a}_\Sigma\right)\right]\Bigg\}
\end{split}
\end{equation*}
for an absolute constant $A$.
\end{theorem}
\begin{proof}
For $\abs{z} \le c$, we have the identity,
\begin{equation*}
    -\int_0^c\left[(z-u)_+ e^{iu} + (-z-u)_+ e^{-iu}\right] du = e^{iz}-iz-1  
\end{equation*}
Taking $c = \norm{\omega}_1$, $z = \omega ^\top x$, $a = a(\omega) = \frac{\omega}{\norm{\omega}_1}$ and $u = \norm{\omega}_1 t$, $t \geq 0$, we have
\begin{equation*}
          -\norm{\omega}_1^2 \int_0^\infty \left[(a^\top x-t)_+ e^{i\norm{\omega}_1 t} + (-a^\top x-t)_+ e^{-i\norm{\omega}_1 t}\right] dt = e^{i\omega^\top x} - i\omega^\top x - 1 
\end{equation*}
Multiplying the above equation by $\mathcal{F}(f)(\omega) = e^{ib(\omega)}\abs{\mathcal{F}(f)(\omega)}$, where $b(\omega)$  is the angle made by $\mathcal{F}(f)(\omega)$ with the real axis, we have, 
\begin{equation*}
      -\norm{\omega}_1^2 e^{ib(\omega)}\abs{\mathcal{F}(f)(\omega)}   \int_0^1 \left[(a^\top x-t)_+ e^{i\norm{\omega}_1 t} + (-a^\top x-t)_+ e^{-i\norm{\omega}_1 t}\right] dt = \mathcal{F}(f)(\omega) (e^{i\omega^\top x} - i\omega^\top x - 1 )  
\end{equation*}
Integrating the above over $\omega\in\R^d$, we have,
\begin{equation}
\label{eq: eq 10.1}
     - \int_{\R^d} \norm{\omega}_1^2 e^{ib(\omega)}\abs{\mathcal{F}(f)(\omega)}   \int_0^1 \left[(a^\top x-t)_+ e^{i\norm{\omega}_1 t} + (-a^\top x-t)_+ e^{-i\norm{\omega}_1 t}\right] dt d\omega =  \int_{\R^d} \mathcal{F}(f)(\omega) (e^{i\omega^\top x} - i\omega^\top x - 1 ) d\omega  
\end{equation}
The R.H.S. of \cref{eq: eq 10.1} can be written as 
\begin{equation}
\label{eq: eq 10.2}
    \begin{split}
        \int_{\R^d} \left[\mathcal{F}(f)(\omega)e^{i\omega^\top x} - \mathcal{F}(f)(\omega)i\omega^\top x  - \mathcal{F}(f)(\omega)\right]d\omega = f(x) - x ^\top \nabla f(0) - f(0)
    \end{split}
\end{equation}
Now, we consider the L.H.S. of \cref{eq: eq 10.1} and show that it satisfies the condition of \href{https://en.wikipedia.org/wiki/Fubini\%27s\_theorem}{Fubini's Theorem}. Consider the expression (from the L.H.S. of~\cref{eq: eq 10.1})
\begin{equation*}
    \begin{split}
      - \int_{\R^d} \norm{\omega}_1^2 e^{ib(\omega)}\abs{\mathcal{F}(f)(\omega)}   \int_0^1 \left[(a^\top x-t)_+ e^{i\norm{\omega}_1 t} + (-a^\top x-t)_+ e^{-i\norm{\omega}_1 t}\right] dt d\omega
    \end{split}
\end{equation*}
We wish to show that 
\begin{equation*}
    - \int_{\R^d} \abs{\norm{\omega}_1^2 e^{ib(\omega)}\abs{\mathcal{F}(f)(\omega)}} \int_0^1 \abs{(a^\top x-t)_+ e^{i\norm{\omega}_1 t} + (-a^\top x-t)_+ e^{-i\norm{\omega}_1 t}} dt d\omega < \infty
\end{equation*}
Now,
\begin{equation*}
\begin{split}
    \int_0^1 \abs{(a^\top x-t)_+ e^{i\norm{\omega}_1 t} + (-a^\top x-t)_+ e^{-i\norm{\omega}_1 t}} dt & \leq \int_0^1 \abs{(a^\top x-t)_+ e^{i\norm{\omega}_1 t}} dt + \int_0^1 \abs{(-a^\top x-t)_+ e^{-i\norm{\omega}_1 t}} dt \\
    & = \int_0^1  (a^\top x-t)_+ \abs{e^{i\norm{\omega}_1 t}} dt + \int_0^1 (-a^\top x-t)_+ \abs{e^{-i\norm{\omega}_1 t}} dt \\
    & = \int_0^1  (a^\top x-t)_+ dt + \int_0^1  (-a^\top x-t)_+ dt
\end{split}
\end{equation*}
Now, we see that, 
\begin{equation*}
\begin{split}
    \int_0^1  (a^\top x-t)_+ dt & \le \int_0^\infty  (a^\top x-t)_+ dt \\
    & = \int_0^\infty  (a^\top x-t) \id\{a^\top x-t \geq 0\} \\
    & = \int_0^\infty  (a^\top x-t) \id\{a^\top x \geq t\} \\
    & = \int_0^{a^\top x} (a^\top x-t) \id\{a^\top x \geq 0\} \\
    & = \left[-\frac{(a^\top x - t)^2}{2}\right]_0^{a^\top x}\id\{a^\top x \geq 0\}\\
    & = \frac{(a^\top x)^2}{2}\id\{a^\top x \geq 0\}\\
    & = \frac{(a^\top x)^2_+}{2}
\end{split}
\end{equation*}
and
\begin{equation*}
\begin{split}
    \int_0^1  (-a^\top x-t)_+ dt & \le \int_0^\infty  (-a^\top x-t)_+ dt \\
    & = \int_0^\infty  (-a^\top x-t) \id\{-a^\top x-t \geq 0\} \\
    & = \int_0^\infty  (-a^\top x-t) \id\{-a^\top x \geq t\} \\
    & = \int_0^{-a^\top x} (-a^\top x-t) \id\{-a^\top x \geq 0\} \\
    & = -\int_0^{-a^\top x} (a^\top x+t) \id\{-a^\top x \geq 0\} \\
    & = -\left[\frac{(a^\top x + t)^2}{2}\right]_0^{-a^\top x}\id\{-a^\top x \geq 0\}\\
    & = -\frac{(a^\top x)^2}{2}\id\{-a^\top x \geq 0\}\\
    & = -\frac{(-a^\top x)^2}{2}\id\{-a^\top x \geq 0\}\\
    & = -\frac{(-a^\top x)^2_+}{2}
\end{split}
\end{equation*}
Thus, we have, 
\begin{equation*}
    \begin{split}
    - \int_{\R^d} \abs{\norm{\omega}_1^2 e^{ib(\omega)}\abs{\mathcal{F}(f)(\omega)}} \int_0^1 \abs{(a^\top x-t)_+ e^{i\norm{\omega}_1 t} + (-a^\top x-t)_+ e^{-i\norm{\omega}_1 t}} dt d\omega \\
    \leq - \left[\frac{(a^\top x)^2_+}{2} -\frac{(-a^\top x)^2_+}{2} \right]\int_{\R^d}\norm{\omega}_1^2\abs{e^{ib(\omega)}}\abs{\mathcal{F}(f)(\omega)} d\omega\\
    =  \left[\frac{(-a^\top x)^2_+}{2} -\frac{(a^\top x)^2_+}{2}\right]\int_{\R^d}\norm{\omega}_1^2\abs{\mathcal{F}(f)(\omega)} d\omega 
    = \left[\frac{(-a^\top x)^2_+}{2} -\frac{(a^\top x)^2_+}{2}\right] v_{f,2}
    \leq \frac{(-a^\top x)^2_+}{2} v_{f,2} \\
    \leq \frac{\abs{-a^\top x}^2}{2} v_{f,2} \leq \frac{\norm{x}_\infty^2}{2}v_{f,2} < \infty
    \end{split}
\end{equation*}
Thus, we can say that the L.H.S. of \cref{eq: eq 10.1} satisfies the condition of \href{https://en.wikipedia.org/wiki/Fubini\%27s\_theorem}{Fubini's Theorem}. So, continuing from \cref{eq: eq 10.1}, we have, 
\begin{equation}
\label{eq: eq 10.3}
     - \int_{\R^d}  \int_0^1 \norm{\omega}_1^2 e^{ib(\omega)}\abs{\mathcal{F}(f)(\omega)} \left[(a^\top x-t)_+ e^{i\norm{\omega}_1 t} + (-a^\top x-t)_+ e^{-i\norm{\omega}_1 t}\right] dt d\omega =  \int_{\R^d} \mathcal{F}(f)(\omega) (e^{i\omega^\top x} - i\omega^\top x - 1 ) d\omega  
\end{equation}
Taking $g(t,\omega) = -[(a^\top x - t)_+ \cos(\norm{\omega}_1t + b(\omega)) + (-a^\top x - t)_+ \cos(\norm{\omega}_1t - b(\omega))]\norm{\omega}_1^2\abs{\F (f)(\omega)}$, we can write, 
\begin{equation*}
\begin{split}
    & \int_{\R^d}\int_0^1 g(t,\omega)dtd\omega = \int_{\R^d}  \mathcal{F}(f)(\omega) (e^{i\omega^\top x} - i\omega^\top x - 1 ) d\omega \\
    \implies & \int_{\R^d}\int_0^1 g(t,\omega)dtd\omega = f(x) - x^\top \nabla f(0) - f(0)
\end{split}
\end{equation*}
Consider the probability measure on $[0,1]\times\R^d$ defined by 
\begin{equation*}
    dP'(t,\omega) = \frac{1}{v}\abs{\cos(\norm{\omega}_1t + b(\omega)}\norm{\omega}_1^2\abs{\F(f)(\omega)}
\end{equation*}
where, $v = \int_{\R^d}\int_0^1[\abs{\cos(\norm{\omega}_1t + b(\omega))} + \abs{\cos(\norm{\omega}_1t - b(\omega))}]\norm{\omega}_1^2\abs{\F(f)(\omega)}dtd\omega \le 2v_{f,2}$. 

Define a function $h(t,a) = (a^\top x - t)_+\eta(t,\omega)$, where, $\eta(t,\omega)= -sgn \cos(\norm{\omega}_1t + b(\omega))$. We note that $h(t,a)(x)$ is of the form $\pm(a^\top x - t)_+ = \eta'(t,\omega)(a^\top x - t)_+$, where, $\eta'(t,\omega) \in \{\pm1\}$. Thus, we see that, 
\begin{equation*}
    \begin{split}
        f(x) - f(0) - \nabla f(0)^\top x & = v \int_{[0,1]\times\R^d} \eta'(t,\omega)(a^\top x - t)_+dP'(t,\omega) \\
        & = v\int_{[0,1]\times\{a: \norm{a} = 1\}}\eta(t,a)(a^\top x - t)_+dP(t,a) \quad \bigg[\because a(\omega) = \frac{\omega}{\norm{\omega}_1}\bigg]
    \end{split}
\end{equation*}
where, $\eta'(t,a(\omega)) = \eta'(t,\omega) \in \{\pm1\}$. 
Thus, we have proved that, $\exists$ a probability measure $P$ on $[0,1]\times \{a: \norm{a} = 1\}$, $\eta \in \{\pm 1\}$, $s = 2$ and $v$ such that $\abs{v} \leq 2v_{f,2}$, such that,
\begin{equation*}
    f(x) - f(0) - x\cdot \nabla f(0) = v \int_{[0,1] \times \{a: \norm{a} = 1\}} \eta(t,a) (a\cdot x - t)_+^{s-1} dP(t,a) \quad \forall x \in D. 
\end{equation*}
Then, we have, 
\begin{equation*}
\begin{split}
    \E f(S_{n,X}) - \E f(Z) & = \E\Bigg[f(0) + \nabla f(0)^\top S_{n,X} + v \int_{[0,1] \times \{a: \norm{a} = 1\}} \eta(t,a) (a\cdot S_{n,X} - t)_+^{s-1} dP(t,a)\Bigg]  - \\
    & \E\Bigg[f(0) + \nabla f(0)^\top Z + v \int_{[0,1] \times \{a: \norm{a} = 1\}} \eta(t,a) (a\cdot Z - t)_+^{s-1} dP(t,a)] \\
    & = \nabla f(0)^\top \E(S_{n,X} - Z) + \E g(S_{n,X}) - \E g(Z) \\ 
    & = \E g(S_{n,X}) = \E g(Z) \\
    & [\because \E S_{n,X} - \E Z \implies \nabla f(0)^\top \E(S_{n,X} - Z) = 0]
\end{split}
\end{equation*}
where, $g$ is a function that admits a representation of the form $g(x) = v \int_{[0,1] \times \{a: \norm{a} = 1\}} \eta(t,a) (a\cdot x - t)_+^{s-1} dP(t,a)$.

Then, we have, 
\begin{equation*}
\begin{split}
       \abs{\E f(S_{n,X}) - \E f(Z)} & = \abs{\E g(S_{n,X}) - \E g(Z)} \\
       & \leq \frac{A\abs{v}}{n}\Bigg\{\sup_{a:\norm{a}_1=1}\left[\frac{1}{\norm{a}_\Sigma}\sum_{k=1}^n \E  (a^\top X_k)^2 \I(\abs{a^\top X_k} \geq \sqrt{n}\norm{a}_\Sigma)\right] + \\ 
      & \sup_{a:\norm{a}_1=1}\left[\frac{1}{2\norm{a}_\Sigma^2}\sum_{k=1}^n\E \frac{\abs{a^\top X_k}^3}{\sqrt{n}}\I(\abs{a^\top X_k} <  \sqrt{n}\norm{a}_\Sigma)\right]\Bigg\} \\
      & [\text{using the bound for }s=2 \text{ in }\cref{bound for functions which have an integral representation of the form given in Theorem 1 of klusowski2018approximation}] \\
      & \leq 2\frac{Av_{f,2}}{n}\Bigg\{\sup_{a:\norm{a}_1=1}\left[\frac{1}{\norm{a}_\Sigma}\sum_{k=1}^n \E  (a^\top X_k)^2 \I(\abs{a^\top X_k} \geq \sqrt{n}\norm{a}_\Sigma)\right] + \\ 
      & \sup_{a:\norm{a}_1=1}\left[\frac{1}{2\norm{a}_\Sigma^2}\sum_{k=1}^n\E \frac{\abs{a^\top X_k}^3}{\sqrt{n}}\I(\abs{a^\top X_k} <  \sqrt{n}\norm{a}_\Sigma)\right]\Bigg\}
\end{split}
\end{equation*}
for an absolute constant $A$. 
\end{proof}


\begin{theorem}
\label{Theorem 3 of klusowski2018approximation}
Suppose $X_1, \dots, X_n$ is a sequence of mean zero independent $d$-dimensional random vectors. Set $S_{n,X} = n^{-1/2}\sum_{i=1}^n X_i$. Let $Z$ be a $d$-dimensional Gaussian random vector with mean zero and variance-covariance matrix given by $\Sigma = \mathrm{Var}(S_{n,X})$. Let $D: [-1,1]^d$. Suppose $f: D \longrightarrow \R$ admits a Fourier representation $f(x) = \int_{\R^d}e^{ix^\top \omega} \mathcal{F}(f)(\omega)d\omega$ and
\begin{equation*}
    v_{f,3} = \int_{\R^d} \norm{\omega}_1^3\abs{\mathcal{F}(f)(\omega)}d\omega < \infty.
\end{equation*}
Then, $\exists$ a probability measure $P$ on $[0,1]\times \{a: \norm{a} = 1\}$, $\eta \in \{\pm 1\}$, $s = 3$ and $v$ such that $\abs{v} \leq 2v_{f,3}$, such that,
\begin{equation*}
    f(x) - f(0) - \nabla f(0)^\top x - \frac{x^\top \nabla_2f(0) x}{2} =  v\int_{[0,1]\times\{a:\norm{a}=1\}}\eta(t,a)(a^\top x - t)^{s-1}_+dP(t,a) \quad \forall x \in D.
\end{equation*}
and hence, we have,
\begin{equation*}
\begin{split}
    \abs{\E f(S_{n,X}) - \E f(Z)}  & \leq 4\frac{Av_{f,3}}{n}\Bigg\{\sup_{a:\norm{a}_1=1}\sum_{k=1}^n \E \left[(a^\top X_k)^2 \ln\left(\frac{e\abs{a^\top X_k}}{\norm{a}_\Sigma}\right)\right]\I(\abs{a^\top X_k} \geq \sqrt{n}\norm{a}_\Sigma) + \\
      & \sup_{a:\norm{a}_1=1}\frac{1}{\norm{a}_\Sigma} \sum_{k=1}^n\E\frac{\abs{a^\top X_k}^3}{\sqrt{n}}\I(\abs{a^\top _k} < \sqrt{n}\norm{a}_\Sigma)  \Bigg\}
\end{split}
\end{equation*}
for an absolute constant A.
\end{theorem}

\begin{proof}
To show that  $\exists$ a probability measure $P$ on $[0,1]\times \{a: \norm{a} = 1\}$, $\eta \in \{\pm 1\}$, $s = 3$ and $v$ such that $\abs{v} \leq v_{f,3}$, such that,$f(x) - f(0) - \nabla f(0)^\top x - \frac{x^\top \nabla_2f(0) x}{2} =  v\int_{[0,1]\times\{a:\norm{a}=1\}}\eta(t,a)(a^\top x - t)^{s-1}_+dP(t,a)$, $\forall x \in D$, we use a technique exactly similar to that used in \cref{Theorem 2 of klusowski2018approximation}.
The function $f(x) - \frac{x^\top\nabla\nabla^\top x f(0)}{2} - x^\top \nabla f(0) - f(0)$ can be written as the real part of
\begin{equation*}
    \int_{\R^d}\Bigg(e^{i\omega^\top x} + \frac{(\omega ^\top x)^2}{2} - i\omega^\top x - 1\Bigg) \F(f)(\omega)d\omega.
\end{equation*}
As before, the above integrand admits an integral representation given by
\begin{equation*}
    \frac{i}{2}\norm{\omega}_1^3 \int_0^1[(-a^\top x - t)_+^2 e^{-i\norm{\omega}_1t} - (a^\top x - t)_+^2 e^{i\norm{\omega}_1t}]dt,
\end{equation*} 
which can be used to show that 
\begin{equation*}
    f(x) - \frac{x^\top\nabla\nabla^\top x f(0)}{2} - x^\top \nabla f(0) - f(0) = \frac{v}{2} \int_{\{-1,1\}\times[0,1]\times\R^d} h(z,t,a)(x) dP(z,t,\omega), 
\end{equation*}
where, 
\begin{equation*}
    h(z,t,a) sgn \sin(z\norm{\omega}_1t + b(\omega)) (za^\top x - t)^2_+
\end{equation*}
and
\begin{equation*}
    dP(z,t,\omega) = \frac{1}{v}\abs{\sin(z\norm{\omega}_1t + b(\omega)} \norm{\omega}_1^3 \abs{\F(f)(\omega)} dt d\omega,
\end{equation*}
\begin{equation*}
    v = \int_{\R^d}\int_0^1[\abs{\sin(\norm{\omega}_1t + b(\omega))}+\abs{\sin(\norm{\omega}_1t - b(\omega))}]\norm{\omega}_1^3\abs{\F(f)(\omega)}dtd\omega \le 2v_{f,3}
\end{equation*} 

Thus, we have proved that, $\exists$ a probability measure $P$ on $[0,1]\times \{a: \norm{a} = 1\}$, $\eta \in \{\pm 1\}$, $s = 3$ and $v$ such that $\abs{v} \leq v_{f,3}$, such that,
\begin{equation*}
    f(x) - f(0) - \nabla f(0)^\top x - \frac{x^\top \nabla_2f(0) x}{2} =  v\int_{[0,1]\times\{a:\norm{a}=1\}}\eta(t,a)(a^\top x - t)^{s-1}_+dP(t,a) \quad \forall x \in D.
\end{equation*}
Then, 
\begin{equation}
\begin{split}
        \E f(S_{n,X}) - \E f(Z) & = \E\Bigg[f(0) + \nabla f(0)^\top S_{n,X} +  \frac{S_{n,X}^\top \nabla_2f(0) S_{n,X}}{2} +   v\int_{[0,1]\times\{a:\norm{a}=1\}}\eta(t,a)(a^\top S_{n,X} -  t)^{s-1}_+dP(t,a)\Bigg] - \\
        & E\Bigg[f(0) + \nabla f(0)^\top Z +  \frac{Z^\top \nabla_2f(0) Z}{2} +   v\int_{[0,1]\times\{a:\norm{a}=1\}}\eta(t,a)(a^\top Z - t)^{s-1}_+dP(t,a)\Bigg] \\
        & = \nabla f(0)^\top \E(S_{n,X} - Z) + \frac{1}{2}\E\left(S_{n,X}^\top \nabla_2f(0) S_{n,X} - Z^\top \nabla_2f(0) Z\right) + \E g(S_{n,X}) - \E g(Z) \\
        & \text{where, } g \text{ a representation of the form } g(x) = v\int_{[0,1]\times\{a:\norm{a}=1\}}\eta(t,a)(a^\top x - t)^{s-1}_+dP(t,a)\\
        & = \E g(S_{n,X}) - \E g(Z)
\end{split}
\end{equation}
Then, we have, 
\begin{equation}
    \begin{split}
       \abs{\E f(S_{n,X}) - \E f(Z)} & = \abs{\E g(S_{n,X}) - \E g(Z)} \\ 
       &  \leq 2\frac{A\abs{v}}{n}\Bigg\{\sup_{a:\norm{a}_1=1}\sum_{k=1}^n \E \left[(a^\top X_k)^2 \ln\left(\frac{e\abs{a^\top X_k}}{\norm{a}_\Sigma}\right)\right]\I(\abs{a^\top X_k} \geq \sqrt{n}\norm{a}_\Sigma) + \\
      & \sup_{a:\norm{a}_1=1}\frac{1}{\norm{a}_\Sigma} \sum_{k=1}^n\E\frac{\abs{a^\top X_k}^3}{\sqrt{n}}\I(\abs{a^\top _k} < \sqrt{n}\norm{a}_\Sigma)  \Bigg\} \quad [\text{using the bound for } s = 3 \text{ in \Cref{bound for functions which have an integral representation of the form given in Theorem 1 of klusowski2018approximation}}]\\
      & \leq  4\frac{Av_{f,3}}{n}\Bigg\{\sup_{a:\norm{a}_1=1}\sum_{k=1}^n \E \left[(a^\top X_k)^2 \ln\left(\frac{e\abs{a^\top X_k}}{\norm{a}_\Sigma}\right)\right]\I(\abs{a^\top X_k} \geq \sqrt{n}\norm{a}_\Sigma) + \\
      & \sup_{a:\norm{a}_1=1}\frac{1}{\norm{a}_\Sigma} \sum_{k=1}^n\E\frac{\abs{a^\top X_k}^3}{\sqrt{n}}\I(\abs{a^\top _k} < \sqrt{n}\norm{a}_\Sigma)  \Bigg\}
    \end{split}
\end{equation}
\end{proof}
\section{Discussion}\label{sec:discussion}
In this paper, we have provided bounds on the difference between functions of random variables using level sets of functions. Using classical uniform and non-uniform Berry--Esseen bounds for univariate random variables. The resulting bounds can be applied to single-layer neural networks and functions on $[-1, 1]^d$ with finite weighted norm integrable Fourier transform. These functions belong to the functions in Barron space. Unlike the classical bounds that depend on the oscillation function of $f$, our bounds do not have an explicit dimension dependence.
In part II, we will explore extensions of these results to functions with integrable Fourier transforms on $\mathbb{R}^d$ and moreover, explore bounds obtained using function approximation theory using radial basis or neural networks.
\paragraph{Acknowledgments.} This work is partially supported by NSF DMS–2113611.

\bibliography{references}

\begin{thebibliography}{}

\bibitem[Angst and Poly, 2017]{angst2017weak}
Angst, J. and Poly, G. (2017).
\newblock A weak cram{\'e}r condition and application to edgeworth expansions.
\newblock {\em Electronic Journal of Probability}, 22(59):1--24.

\bibitem[Bentkus, 2003a]{bentkus2003new}
Bentkus, V. (2003a).
\newblock A new method for approximations in probability and operator theories.
\newblock {\em Lithuanian Mathematical Journal}, 43:367--388.

\bibitem[Bentkus, 2003b]{bentkus2003dependence}
Bentkus, V. (2003b).
\newblock On the dependence of the berry--esseen bound on dimension.
\newblock {\em Journal of Statistical Planning and Inference}, 113(2):385--402.

\bibitem[Bentkus, 2004]{bentkus2005lyapunov}
Bentkus, V. (2004).
\newblock A {L}yapunov type bound in {${\bf R}^d$}.
\newblock {\em Teor. Veroyatn. Primen.}, 49(2):400--410.

\bibitem[Bhattacharya and Rao, 2010]{bhattacharya2010normal}
Bhattacharya, R.~N. and Rao, R.~R. (2010).
\newblock {\em Normal approximation and asymptotic expansions}.
\newblock SIAM.

\bibitem[Klusowski and Barron, 2018]{klusowski2018approximation}
Klusowski, J.~M. and Barron, A.~R. (2018).
\newblock Approximation by combinations of relu and squared relu ridge
  functions with $\ell^{} 1$ and $\ell^{} 0$ controls.
\newblock {\em IEEE Transactions on Information Theory}, 64(12):7649--7656.

\bibitem[Rai{\v{c}}, 2019]{raivc2019multivariate}
Rai{\v{c}}, M. (2019).
\newblock A multivariate berry--esseen theorem with explicit constants.
\newblock {\em Bernoulli}, 25(4A):2824--2853.

\bibitem[Sazonov, 1981]{Sazonov}
Sazonov, V.~V. (1981).
\newblock {\em Normal approximation---some recent advances}, volume 879 of {\em
  Lecture Notes in Mathematics}.
\newblock Springer-Verlag, Berlin-New York.

\bibitem[Shevtsova, 2020]{shevtsova2020lower}
Shevtsova, I. (2020).
\newblock Lower bounds for the constants in non-uniform estimates of the rate
  of convergence in the clt.
\newblock {\em Journal of Mathematical Sciences}, 248(1):92--98.

\bibitem[Zhilova, 2020]{zhilova2020new}
Zhilova, M. (2020).
\newblock New edgeworth-type expansions with finite sample guarantees.
\newblock {\em arXiv preprint arXiv:2006.03959}.

\end{thebibliography}
\bibliographystyle{apalike}
\end{document}